\documentclass{archiclas}

\usepackage{epsfig}
\usepackage{amsmath}
\usepackage{amssymb}
\usepackage{psfrag} 
\usepackage{version} % pour utiliser {comment}
\usepackage[latin1]{inputenc}
\def\dsp{\displaystyle}

\def\R{\mathbb{R}}
\def\N{\mathbb{N}}
\def\De{\mathcal{D}e}

\def\0{{\bf 0}}

\def\E{{\bf E}}
\def\F{{\bf F}}
\def\G{{\bf G}}
\def\Q{{\bf Q}}
\def\u{{\bf u}}
\def\W{{\bf W}}
\def\x{{\bf x}}

\def\eps{\varepsilon}
\def\bkappa{\boldsymbol{\kappa}}
\def\bnu{\boldsymbol{\nu}}
\def\fy{\varphi}
\def\tphi{\widetilde{\phi}}
\def\tpsi{\widetilde{\psi}}
\def\btheta{\boldsymbol{\theta}}

\def\div{\mbox{\,{\textrm{div}}}}
\renewcommand{\nfrac}[2]{#1/#2}

\title
{Fokker-Planck equation in bounded domain}

\alttitle{Equation de Fokker-Planck dans un domaine borné}

\author{\firstname{Laurent} \lastname{Chupin}}

\address{Université de Lyon\\
INSA de Lyon - Pôle de Mathématiques\\
CNRS, UMR5208, Institut Camille Jordan\\
21~av. Jean Capelle, 69621 Villeurbanne Cedex, France}

\email{laurent.chupin@insa-lyon.fr}

  
\altkeywords{Equation de Fokker-Planck, Domaine borné, Confinement, Mécanique des fluides, Ecoulement de polyméres}


\begin{document}
%% Abstract 
\begin{abstract}
We study the existence and the uniqueness of a solution~$\fy$ to the linear Fokker-Planck equation $-\Delta \fy + \div(\fy \, \F) = f$ in a bounded domain of $\R^d$ when $\F$ is a ``confinement'' vector field acting for instance like the inverse of the distance to the boundary. An illustration of the obtained results is given within the framework of fluid mechanics and polymer flows.
\end{abstract}
%
%% French abstract
\begin{altabstract}
On étudie l'existence et l'unicité de solution~$\fy$ à l'équation de Fokker-Planck linéaire $-\Delta \fy + \div(\fy\, \F) = f$ sur un domaine borné de~$\R^d$ lorsque~$\F$ est un champ de vecteurs ``confinant'' par exemple comme l'inverse de la distance au bord. Une illustration des résultats obtenus est donnée dans le cadre de la mécanique des fluides et des écoulements de polymères.
\end{altabstract}
\maketitle
%
%%%%%%%%%%%%%%%%%%%%%%%%%%%%%%%%%%%%%%%%%%%%%%%%%%%%%%%%%%%%%%%%%%%%%%%%%%%%%%
\section{Introduction}\label{sec-introduction}
%%%%%%%%%%%%%%%%%%%%%%%%%%%%%%%%%%%%%%%%%%%%%%%%%%%%%%%%%%%%%%%%%%%%%%%%%%%%%%
%
\noindent In this paper we are interested in the so called Fokker-Planck equation 
\begin{equation}\label{FP}
-\Delta \fy + \div(\fy \, \F) = f.
\end{equation}
$\bullet$ In the simplest case (that is $\F=\0$) this equation is known as the Laplace equation (when $f=0$) or as the Poisson equation (when $f\neq 0$). The solutions of these equations are important in many fields of science, notably the fields of electromagnetism, astronomy and fluid dynamics, because they describe the behavior of electric, gravitational and fluid potentials.\\
$\bullet$ More generally, the main reason of the physical interest of equation~\eqref{FP} comes from the fact that it can be put in conservative form $\div(J)=f$ with $J=-\nabla \fy + \fy \, \F$. Thus it can be connected to a generalization to the Fick's law $J=-\nabla\fy$ connecting diffusion flux~$J$ and concentration~$\fy$ in inhomogeneous environments, see~\cite{Escande,Sattin}.\\
$\bullet$ In the dynamical systems framework (see for instance~\cite{Zeeman}) the non-stationary Fokker-Planck equation $\partial_t \fy = \eps \Delta \fy - \div(\fy \, \F)$ is usually introduced. In this case, the function~$\fy$ represents the smooth probability density of a population driven by~$\F$ and subject to $\eps$-small diffusion in the following sense.
The term $\fy \, \F$ is a vector field representing the population~$\fy$ moving with the flow of~$\F$, and so the divergence of this vector field represents a thinning out of the population due to~$\F$, which therefore contributes negatively to the local growth rate of the population,~$\partial_t\fy$. This explains the drive term. Meanwhile the term $\eps \Delta \fy$ represents $\eps$-small diffusion, and contributes positively to the growth rate. The study which is presented here concerns in particular the existence and the uniqueness of a steady-state solution. We note that the theory is closely related to applications, because the steady-state~$\fy$ is an $\eps$-smoothing of the measure on the attractors of the flow of~$\F$ (see~\cite{Zeeman}) and therefore in numerical and physical experiments~$\fy$ can be used to model the data with $\eps$-error.\\
$\bullet$ According to the contexts, the vector field~$\F$ can take various forms. In particular it may occur that physically realistic assumptions do not make it possible to conclude only with the already known results. We will give such a caricatural example in the last part.\\[0.3cm]
{\it Besides the problems of the existence and the uniqueness, the question which interests us is to know which boundary conditions are needed to ensure the existence and the uniqueness of a solution of equation~\eqref{FP} in the bounded case.}\\[0.3cm]
We will see that this depends on~$\F$. When $\F$ is regular enough, i.e. does not diverge too quickly at the boundary, data on~$\fy$ at the boundary of the domain enable to ensure the uniqueness of the solution. We remind of this result at the beginning (in particular because the proof resembles ours). We also say that when the domain does not have boundary, for instance if we are interested in the space~$\R^d$ or on a compact variety without boundary, uniqueness is ensured by imposing the average of~$\fy$.\\[0.3cm]
{\it We prove that, in the bounded case, when $\F$ is not so regular, the "good" condition to ensure uniqueness is still to impose the average of~$\fy$, and that in that case, the unique solution vanishes on the boundary.}

\mathversion{bold}
\subsection{Some known results on equation~\eqref{FP}}\label{some-known}
\mathversion{normal}
%%%%%%%%%%%%%%%%%%%%%%%%%%%%%%%%%%%%%%%%%%%%%%%%%%%%%%%%%%%%%%%%%%%%%%%%%%%%%%
%
\noindent Except for the case where $\F=\0$, a particularly simple case corresponds to $\F=\nabla V$ ($V$ is assumed to be regular and differentiable) and $f=0$. In this case $\fy = \exp(V)$ is a solution of~\eqref{FP} in the bounded case as well as in the compact case or in the unbounded case. 
In the same way, we can easily prove that the case $\F= \nabla V + \G$ admits the solution $\fy = \exp(V)$ if and only if $\div(\G) + \G \cdot \nabla V = 0$. 
Up to a renormalization, the average of such a solution will be equal to~$1$, as soon as $\exp(V)\in L^1(\Omega)$. 
%Plus généralement, les champs de la forme $\F=\nabla V$, $V$ différentiable tels que $\exp(V)\in L^1(\Omega)$ sont appelés des champs de confinement (voir \cite{Arnold}).
The average value is consequently an essential ingredient to have uniqueness of the solution of equation~\eqref{FP} and we could be interested in the following problem
\begin{equation}\label{pb-zeeman}
\left\{
\begin{aligned}
-\Delta \fy + \div(\fy \, \F) &= 0 \qquad \text{in $\Omega$},\\
\text{with}\quad \int_\Omega \fy &= 1.
\end{aligned}
\right.
\end{equation}
In the compact case without boundary E.C. Zeeman~\cite{Zeeman} proves the existence and uniqueness of a solution~$\fy$ for an arbitrary smooth vector field~$\F$ (and without term source $f=0$) on a compact manifold by the Perron-Fröbenius method:
\begin{theo}
Let~$\Omega$ be a compact manifold without boundary.\\
If~$\F\in \mathcal C^\infty(\Omega)$ then there exists a unique non negative solution of~\eqref{pb-zeeman}.
\end{theo}
\noindent Without proof in the non-compact case (as writes E.C. Zeeman on p.~152, the extension of such results to non-compact case is an open question), E.C. Zeeman gives some example with $\Omega=\R^d$. These examples show that in the unbounded case uniqueness follows from some ``boundary'' conditions on~$\F$, which are given by the behavior of~$\F$ outside large sphere.
% of the form
%    \begin{equation*}
%      \exists (a,b)\in \R^*_+ \quad \forall \x\in \R^d \qquad |\x|>b \Longrightarrow |\F(\x)| < a|\x|.
%    \end{equation*}
Many other works concern these equations of the Fokker-Planck type in~$\R^d$. Most of these works describe specific assumptions for the potential~$\F$ at infinity. Let us quote as an example the beautiful recent series of works by Hérau, Nier and Helffer~\cite{Helffer,Herau} about the linear kinetic Fokker-Planck equation. See also the article~\cite{Noarov} by Noarov in wich the author gives some smallness conditions in some norm and rapid decay at infinity for~$\F$ to ensure the existence of a solution other than identical zero (in the case $f=0$).\\[0.3cm]%Notice that the condition of smallness does not contain derivatives of~$\F$ , and the smoothness of~$\F$ is not assumed.
In this article, we are interested in a possible generalization in bounded domains. Usually, such a problem is coupled with boundary conditions (Dirichlet, Neumann or mixted boundary conditions). For instance, the natural weak formulation of the problem
\begin{equation*}%\label{strong-droniou}
\left\{
\begin{aligned}
-\Delta \fy + \div(\fy \, \F) &= f \qquad \text{in $\Omega$},\\
\text{with}\quad \fy &= 0 \qquad \text{on $\partial \Omega$}
\end{aligned}
\right.
\end{equation*}
is written
\begin{equation}\label{weak-droniou}
\left\{
\begin{aligned}
& \text{Find $\fy\in H^1_0(\Omega)$ such that for all $\psi \in H^1_0(\Omega)$}\\
&\hspace{0cm} \int_\Omega \nabla \fy \cdot \nabla \psi
- \int_\Omega \fy \, \F \cdot \nabla \psi
= \langle f,\psi \rangle,
\end{aligned}
\right.
\end{equation}
where $\langle f,\psi \rangle$ corresponds to the duality product $(H^{-1}(\Omega), H^1_0(\Omega))$.
In order that all terms in \eqref{weak-droniou} be defined, the minimal hypotheses on data are: $f\in H^{-1}(\Omega)$ and, thanks to the classical Sobolev injections, $\F\in L^{d_*}(\Omega)$ where $d_*=d$ for $d\geq 3$ and $d_*\in \,]2,+\infty[$ for $d=2$. Within this framework, we have (see \cite{Droniou}):
\begin{theo}\label{th-Droniou}
Let $\Omega$ be a bounded domain of $\R^d$, $d\geq 2$.\\
If $f\in H^{-1}(\Omega)$ and $\F\in L^{d_*}(\Omega)$ then there exists a unique solution of~\eqref{weak-droniou}.
\end{theo}
\noindent The proof of the generalization which we propose is primarily based on the proof of this theorem. The main difficulty which appears for the study of problem~\eqref{weak-droniou} is the following: although the operator $\Delta$ is coerciv, the operator $-\Delta + \div(\cdot \, \F)$ is generally not coerciv. The reason for which the result is still valid lies in the (conservative) form of the term $\div (\fy\, \F)$. Let us note that an equivalent theorem can be proved (see \cite{Droniou}) for equations of kind $-\Delta \fy + \G \cdot \nabla \fy = f$ but that it is not possible to obtain a similar general result for equations of the type $-\Delta \fy + \div(\fy \, \F) + \G \cdot \nabla \fy = f$. In fact, the sum $\div(\fy \, \F) + \G \cdot \nabla \fy$ makes appear zeroth order terms and it is well known that the solutions of $-\Delta \fy + \lambda\,\fy = 0$, $\lambda \in \R$, with homogeneous Dirichlet boundary condition are not unique.\\
Moreover, J.~Droniou proved that the same result is valid for other boundary conditions as non-homogeneous Dirichlet, Fourier or mixed boundary conditions (using more regularity for the domain $\Omega$, say with Lipschitz continuous boundary). Concerning the Neumann boundary conditions, J.~Droniou and J.-L.~Vazquez recently showed that the same problem admits, for each fixed mean value, a unique solution with the said mean value (see~\cite{Droniou1}).\\
An other question is debated in the present paper: Which necessary and sufficient conditions must be placed on~$f$, and what are the degrees of freedom on the solutions if $\F$ is not regular enough~?

\subsection{An (partial) answer}
%%%%%%%%%%%%%%%%%%%%%%%%%%%%%%%%%%%%%%%%%%%%%%%%%%%%%%%%%%%%%%%%%%%%%%%%%%%%%%

\noindent We show in this article that if the normal component of the vector~$\F(\x)$ behaves
\footnote
{
We can verify that we have $\F\notin L^{d_*}(\Omega)$, and that consequently the announced result is a generalization of the Theorem~\ref{th-Droniou}. 
%It will be also noticed that the condition $\alpha>1$ is exactly the condition $\exp(V) \in L^1(\Omega)$ when $\F= \nabla V$ which indicates that the vector field $\F$ is in this case a ``confinement'' field.
} 
like 
%$$\nabla (\ln ( \mathrm{dist}(\x,\partial \Omega) ^\alpha) ) $$
$\frac{\alpha}{\mathrm{dist}(\x,\partial \Omega)}$
in a neighborhood of the boundary $\partial \Omega$ with $\alpha>1$ then there exists a unique solution~$\fy$ to the Fokker-Planck equation~\eqref{FP} as soon as the average of~$\fy$ is given. Moreover, we can show that this unique solution automatically vanishes at the boundary~$\partial \Omega$.
More precisely we prove
\begin{theo}
Let $\Omega$ be a bounded domain of $\R^d$, $d\geq 2$.\\ 
%We denote by $\Gamma$ its boundary which is assumed to be of class $\mathcal C^2$.
Let $f\in H^{-1}_M$ (this space will be precised later) and $\F=\bkappa+\nabla V$ where $\bkappa\in L^\infty(\Omega)$ and $V\in \mathcal C^\infty(\Omega)$ satisfies $V = -\infty$ on the boundary of~$\Omega$. 
Under assumptions~\eqref{H1}, \eqref{H2} and~\eqref{H3} (see more details page~\pageref{the-main}) then there exists a unique solution of the Fokker-Planck equation~\eqref{FP} such that $\int_\Omega \fy = 1$.
\end{theo}
\noindent Obviously, the additive assumptions~\eqref{H1}, \eqref{H2} and~\eqref{H3} enable to take into account the examples where the normal part of~$\F(\x)$ behaves like $\frac{\alpha}{\mathrm{dist}(\x,\partial \Omega)}$ in a neighborhood of the boundary $\partial \Omega$ with $\alpha>1$.\\
They are not satisfied when $\alpha\leq 1$. Thus, there remain many cases without answers: for instance, when the normal part of~$\F$ behaves like $\frac{\alpha}{\mathrm{dist}(\x,\partial \Omega)}$ with $\alpha \leq 1$, and when the normal part of~$\F$ behaves like $\frac{\alpha}{\mathrm{dist}(\x,\partial \Omega)^\beta}$ with $\alpha >0$ and $\beta > 1$.\\
% Nevertheless, in the one dimensional case $d=1$ all explicit calculations can be carried out and the theorem is not satisfied when $\alpha\leq 1$. The first example of the following part shows this fact.\\
%On the other hand in the one dimensional case, the previous theorem is true for more ``general'' potentials, for example for potentials behaving like $\frac{\alpha}{\mathrm{dist}(\x,\partial \Omega)^2}$, see the second example of the next section. It seems that the proof brought here cannot easily adapt to such cases in higher dimension.\\
The assumptions~\eqref{H1}, \eqref{H2} and~\eqref{H3} on the potential~$V$ are rather difficult to apprehend. The reason for which we have these assumptions is the following: they are used in this form in each steps of the proof (primarily in the various lemmas). Thus if one of the steps of the proof can be shown in another way that presented here, we can hope to be freed from certain corresponding assumptions.
%Lastly, it should be noted that the potential~$V$ is only used to characterize the behavior of the vector field $\F$ near to the boundary of the domain. The addition of any function $\bkappa \in L^\infty(\Omega)$ can be considered.
%
\subsection{Outline of the paper}
%%%%%%%%%%%%%%%%%%%%%%%%%%%%%%%%%%%%%%%%%%%%%%%%%%%%%%%%%%%%%%%%%%%%%%%%%%%%%%

\noindent The paper is organized as follows:\\
$\bullet$ In Section~\ref{sec-tools}, we give mains tools adapted to the studied problem. First of all, some tools about differential geometry to understand ``explosive'' boundary conditions. Next we give all the lemmas which are used in the main proof.\\
$\bullet$ In Section~\ref{sec-assumptions} the precise statement of the main result is enunciated, see Theorem~\ref{the-main}, page~\pageref{the-main}.\\
$\bullet$ The Section~\ref{sec-proof} is devoted to the proof of Theorem~\ref{the-main}. It is composed of two parts: the existence proof and the uniqueness proof.\\
$\bullet$ The last section (Section~\ref{sec-numeric}) gives an application to fluid mechanics and some numerical results.

%%%%%%%%%%%%%%%%%%%%%%%%%%%%%%%%%%%%%%%%%%%%%%%%%%%%%%%%%%%%%%%%%%%%%%%%%%%%%%
%%%%%%%%%%%%%%%%%%%%%%%%%%%%%%%%%%%%%%%%%%%%%%%%%%%%%%%%%%%%%%%%%%%%%%%%%%%%%%
\section{Main implements}\label{sec-tools}
%%%%%%%%%%%%%%%%%%%%%%%%%%%%%%%%%%%%%%%%%%%%%%%%%%%%%%%%%%%%%%%%%%%%%%%%%%%%%%
%%%%%%%%%%%%%%%%%%%%%%%%%%%%%%%%%%%%%%%%%%%%%%%%%%%%%%%%%%%%%%%%%%%%%%%%%%%%%%

\noindent In~\cite{Chupin}, the author gives an existence and a uniqueness result for a Fokker-Planck equation for a \underline{particular} vector field~$\F$ and in a \underline{particular} domain~$\Omega$ which is a ball. For more complex domains, we must understand the effect of the geometry in the proof. We will present in this part some elements of differential geometry adapted for our calculus. Next, we will precise the functional framework adapted to the Fokker-Planck equation of this paper. Finally, we will give multiple fundamental lemmas which are use for the proof of the Theorem~\ref{the-main}.

\subsection{Elementary differential geometry}\label{subsection_geometry}
%%%%%%%%%%%%%%%%%%%%%%%%%%%%%%%%%%%%%%%%%%%%%%%%%%%%%%%%%%%%%%%%%%%%%%%%%%%%%%

\noindent The results of this part are largely inspired on Subsection 2.1 of the paper~\cite{Boyer} and on the annexe C of the book~\cite{Boyer-Fabrie}.\\
Let~$\Omega$ be a smooth (say $\mathcal C^2$) bounded domain in~$\R^d$, $d\geq 2$. We denote by~$\Gamma$ its boundary and by~$\bnu$ the outward unitary normal to $\Gamma$. The distance between any $\x\in \R^d$ and the boundary~$\Gamma$ is denoted by $\delta_\Gamma(\x)$.\\
For any $\eps\geq 0$, we introduce the open subset of $\Omega$:
\begin{equation*}
  \mathcal O_\eps = \{\x\in \Omega\, ;\, \delta_\Gamma(\x)<\eps\}
\quad \text{and} \quad
  \Omega_\eps = \{\x\in \Omega\, ;\, \delta_\Gamma(\x)>\eps\}.
\end{equation*}
It is classical that, for~$\eps$ small enough, the two maps $\delta_\Gamma$ (called distance to~$\Gamma$) and $P_\Gamma$ (called projection on~$\Gamma$) exist and are regular on $\overline{\mathcal O_{\eps}}$. This allow to use tangential and normal variable near~$\Gamma$, defining for any function $f:\overline{\mathcal O_{\eps}} \rightarrow \R$ the corresponding function $\widetilde f:[0,\eps]\times \Gamma \rightarrow \R$ by\footnote{For non-ambiguous cases,~$\widetilde f$ and~$f$ will be together denote~$f$.}
\begin{equation*}
  \forall (r,\btheta)\in [0,\eps]\times \Gamma\, ,\quad \widetilde f(r,\btheta) = f\big(\btheta-r\, \bnu(\btheta)\big).
\end{equation*}
%
%----------------------------------------
\begin{figure}[htbp]
\begin{center}
{\psfrag{Gamma}{$\Gamma$}\psfrag{x}{$\x$}\psfrag{theta}{$\btheta$}\psfrag{Px}{$P_\Gamma(\x)$}\psfrag{Omegaeps}{$\Omega_\eps$}\psfrag{Oeps}{$\mathcal O_\eps$}\psfrag{dx}{$\delta_\Gamma(\x)$}\psfrag{eps}{$\eps$}\psfrag{nu}{$\bnu(\btheta)$}
\includegraphics[width=8cm]{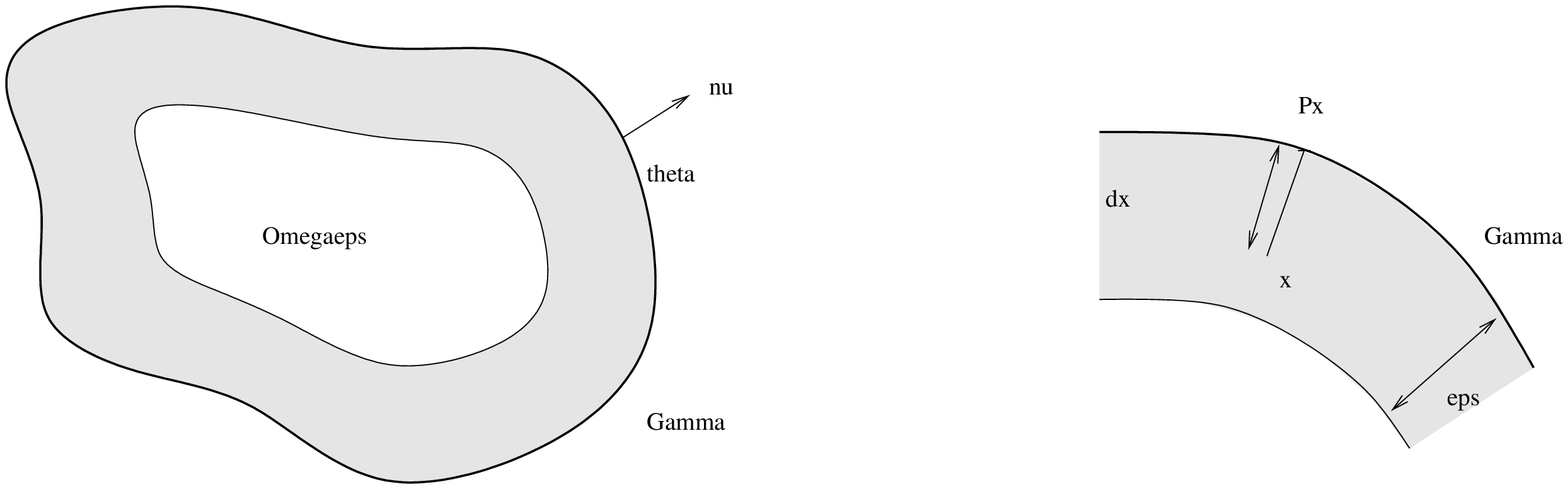}
}
%\caption{}\label{}
\end{center}
\end{figure}
%----------------------------------------

\noindent Moreover, for any $f\in L^1(\mathcal O_\eps)$ we have the following change of variables formula:
\begin{equation}\label{change_of_variables}
  \int_{\mathcal O_\eps} f(\x)\, d\x
=
  \int_\Gamma \int_0^\eps \widetilde f(r,\btheta)\, |J(r,\btheta)|\, dr\, d\btheta,
\end{equation}
where $dr$ and $d\btheta$ corresponds to the Lebesgue measure on $[0,\eps]$ and $\Gamma$ respectively, and where~$J$ is the jacobian determinant of the previous change of variables.
Introduce $\kappa_M = \sup_{\btheta\in \Gamma}\sup_{i} \kappa_i(\theta)$, where $\kappa_1(\btheta)$,..., $\kappa_{d-1}(\btheta)$ are the principal curvatures of~$\Gamma$ at~$\btheta$, we have for $0\leq r\leq \eps$ and for $\btheta\in \Gamma$ 
\begin{equation*}
  (1-\eps\, \kappa_M)^{d-1} \leq J(r,\btheta) \leq 1.
\end{equation*}
Using the change of variables~\eqref{change_of_variables} we deduce that for $f\in L^1(\mathcal O_\eps)$ we have
\begin{equation}\label{change_of_variables_approx}
(1-\eps\kappa_M)^{d-1} \hspace{-0.2cm} \int_\Gamma \hspace{-0.1cm}\int_0^\eps \hspace{-0.1cm} \widetilde f(r,\btheta)\, dr d\btheta
\leq
  \int_{\mathcal O_\eps} \hspace{-0.2cm} f(\x)\, d\x
\leq
  \int_\Gamma \hspace{-0.1cm} \int_0^\eps \hspace{-0.1cm} \widetilde f(r,\btheta)\, dr d\btheta.
\end{equation}
Roughly speaking these relations show us that for thin tubular neighborhood of~$\Gamma$ in~$\Omega$ the jacobian $J(r,\btheta)$ is equivalent to~$1$ for small $r$, uniformly with respect to $\btheta\in \Gamma$. Consequently, the relations~\eqref{change_of_variables_approx} give an approximation of $\int_{\mathcal O_\eps} f$ by $\int_\Gamma \int_0^\eps \widetilde f$.\\
Notice too that it is possible to define~$\delta_\Gamma$ as a regular function on~$\overline{\Omega}$. This extension will be also noted~$\delta_\Gamma$.

\subsection{Functional spaces}
%%%%%%%%%%%%%%%%%%%%%%%%%%%%%%%%%%%%%%%%%%%%%%%%%%%%%%%%%%%%%%%%%%%%%%%%%%%%%%

\noindent The Fokker-Planck equation~\eqref{FP} make appear a potential~$V$, assume to be smooth on~$\Omega$, {\it via} the force $\F=\bkappa + \nabla V$. This potential is supposed to be confinant so that we assume that $e^V\in L^1(\Omega)$. We can always define a maxwellian function~$M$ by
\begin{equation*}
M = \frac{e^V}{\int_\Omega e^V}.
\end{equation*}
Notice that $M\in \mathcal C^\infty(\Omega:\R^*_+)$ with $\int_\Omega M = 1$.
Notice too that the interesting cases for the present study correspond to a maxwelian~$M$ vanishing on the boundary $\Gamma=\partial \Omega$ (that is $V=-\infty$ on $\Gamma$).
All the maxwellians considered in this paper will satisfy $M=0$ on~$\Gamma$.
The Fokker-Planck equation~\eqref{FP} is written
\begin{equation}\label{FP1}
\div ( \fy \bkappa - M\nabla \Big( \frac{\fy}{M} \Big)) = f.
\end{equation}
From the peculiar form of this Fokker-Planck equation~\eqref{FP1}, the adapted functional spaces use Sobolev weight spaces on the domain $\Omega\subset \R^d$. More precisely, we introduce
\begin{equation*}
  L^2_M := M L^2(M dx)
\quad \text{and} \quad
  H^1_M := M H^1(M dx),
\end{equation*}
endowed with their natural norms respectively given by
\begin{equation*}
\begin{aligned}
& \|\fy\|_{L^2_M}^2 = \int_\Omega M\Big|\frac{\fy}{M}\Big|^2,\\
& \|\fy\|_{H^1_M}^2 = \int_\Omega M\Big|\frac{\fy}{M}\Big|^2 + M\Big|\nabla \Big( \frac{\fy}{M} \Big) \Big|^2.
\end{aligned}
\end{equation*}
Note that a large literature on weighted Sobolev spaces exists, see for instance~\cite{stredulinsky,turesson}, the references given in the notes at the end of Chapter 1 of~\cite{HKM} or the classical book~\cite{Triebel}.
Among all the weights which are generally considered in the literature, only some are vanished on~$\R^d$ as it is the case of the weight~$M$.
We refer to~\cite{Triebel} for references of spaces with weights which equal to a power of the distance to the boundary.
Nevertheless, the spaces~$L^2(M dx)$ and~$H^1(M dx)$ being spaces with "traditional" weights, the spaces~$L^2_M$ and~$H^1_M$ are not it.
For this reasons, the results shown in the following section are not always "traditional" corollaries (except some whose proof is a direct consequence of the results in~$L^2(M dx)$ and~$H^1(M dx)$, see for example the proof of lemma~\ref{lem-compacity}).\par
\noindent Note first that the two spaces~$L^2_M$ and~$H^1_M$ are Hilbert spaces (see~\cite[Th.~3.2.2a]{Triebel}).
%--------------------------------------
\noindent We introduce the following normalized subspace
\begin{equation*}
H^1_{M,0} := \{\psi\in H^1_M \ ;\ \int_\Omega \psi = 0\}.
\end{equation*}
In the sequel, the space $H^1_M$ will by equipped too with the semi-norm defined by
$$
\|\psi\|_{H^1_{M,0}}^2 = \int_\Omega M\Big| \nabla \Big( \frac{\psi}{M} \Big) \Big|^2,
$$
and we will see (Lemma~\ref{lem-kernel} on page~\pageref{lem-kernel}) that $\|\cdot\|_{H^1_{M,0}}$ is a norm on the space~$H^1_{M,0}$.
%-------------------------------------
Finally, we denote $H^{-1}_M$ the topological dual of $H^1_{M,0}$, that is the set of continuous linear forms on $H^1_{M,0}$. Each application $\chi\in H^{-1}_M$ will be defined by $\chi:\fy\in H^1_{M,0} \mapsto \langle \chi,\fy \rangle \in \R$. By its continuity, for each $\chi\in H^{-1}_M$ there exists $C\in \R$ such that
\begin{equation*}
\forall \fy \in H^1_{M,0} \quad |\langle \chi,\fy \rangle| \leq C \|\fy\|_{H^1_{M,0}}.
\end{equation*}
As it is usual, the smallest of these constants~$C$ is denoted $\|\chi\|_{H^{-1}_M}$: it is the norm of~$\chi$ on $H^{-1}_M$.

\subsection{Properties of the functional spaces}
%%%%%%%%%%%%%%%%%%%%%%%%%%%%%%%%%%%%%%%%%%%%%%%%%%%%%%%%%%%%%%%%%%%%%%%%%%%%%%

\noindent The proof of the main theorem (Theorem~\ref{the-main}) follows the ideas of J.~Droniou~\cite{Droniou}. Nevertheless, the proof of J.~Droniou, given in the case where $\F \in L^{d_*}(\Omega)$, does not use these ``degenerated'' spaces and use traditional results concerning usual Sobolev spaces. The essential contributions which are presented ties in the fact that these ``classical'' lemmas in the case where~$M$ does not vanished on~$\overline \Omega$ are still true (sometimes in a weaker form) when~$M$ is a maxwellian as previously introduced, and in particular when $M=0$ on~$\Gamma$. 
So, the goal of this subsection is to give some essential properties of these functional spaces (Poincaré-type inequality, Sobolev injection, compacity result, Hardy-type inequality...).\\[0.3cm]
Notice that in the estimates, the symbol~$\lesssim$ means ``up to a harmless multiplicative constant'', allowed to depend on the domain~$\Omega$ only.\\[0.3cm]
The first result that we present is a result allowing to controlled $\frac{1}{\delta_\Gamma}\frac{\fy}{\sqrt M}$ in~$L^2(\Omega)$ as soon as $\fy \in H^1_M$. This result can be seen as an inequality of Hardy-type
\footnote{
Hardy inequality indicates that for a function~$f$ defined on a bounded domain~$\Omega$ of~$\R^d$ and vanishing on the boundary~$\Gamma$ of~$\Omega$ we get $| \nfrac{f}{\delta_\Gamma} |_{L^2(\Omega)} \lesssim |\nabla f|_{L^2(\Omega)}$ where $\delta_\Gamma$ corresponds to the distance to the boundary~$\Gamma$.
}.
We will prove this first lemma under the following assumption
\footnote{
Recall that, as it is specified just before, in this article the function~$M$ is a Maxwellian function which vanished on the boundary of the domain. The assumptions introduced here are consequently additional assumptions.
}
\begin{equation}\label{H1}\tag{$\mathcal H_1$}
\left\{
\begin{aligned}
& \exists a<1 \quad \Big( \frac{\nabla_RM}{M} \Big)^2 + 2 \, \nabla_R\Big( \frac{\nabla_RM}{M} \Big)  \geq \frac{-a}{\delta_\Gamma^2}, \\
& \nabla_R M(0)=0, \\
& \exists b>0 \quad \nabla_R M \int_\Omega \frac{1}{M} < b.
\end{aligned}
\right.
\end{equation}
These hypotheses will be only used in the neighbourhood of the boundary~$\Gamma$ of~$\Omega$. In such neighbourhood, the notation~$\nabla_R$ corresponds to the radial derivative (that is to say the derivative in the direction of the normal vector to the boundary).\par
\noindent The various assumptions introduced in this part will be discussed in Section~\ref{sec-assumptions}. Nevertheless, it is important to notice that the three assumptions formulated in~\eqref{H1} are independent. For example, in the radial case the function~$M$ defined by $M(r)=\sqrt{r}$ satisfy the last point of~\eqref{H1} but does not satisfy the second point. Reciprocally, $M(r)=r/ln(r)$ satisfy the second point but does not satisfy the last point.% A comparer avec Wu-Zhang (4.6)-(4.7)
%
%------------------------------------
\begin{lemma}[Hardy-type inequality]\label{lem-hardy}
If~$M$ satisfies~\eqref{H1} then for any $\fy \in H^1_M$ we have
\begin{equation*}
\int_\Omega \frac{1}{\delta_\Gamma^2}\, M \Big| \frac{\fy}{M} \Big|^2 \lesssim \|\fy\|_{H^1_M}^2.
\end{equation*}
\end{lemma}
%
%------------------------------------
\noindent In fact, as can be seen from the proof, we shall not need entire norm of~$\fy$ in~$H^1_M$ since we will only use the radial part of the gradient.

%------------------------------------
\begin{proof} For any~$\eps$ small enough (more precisly, $\eps\leq \eps_\Omega$, see Part~\ref{subsection_geometry}) we have $\Omega = \overline{\Omega_\eps} \cup \mathcal O_\eps$ and
\begin{equation*}
\int_\Omega \frac{M}{\delta_\Gamma^2}\, \Big| \frac{\fy}{M} \Big|^2
=
\int_{\Omega_\eps} \frac{M}{\delta_\Gamma^2}\, \Big| \frac{\fy}{M} \Big|^2
+
\int_{\mathcal O_\eps} \frac{M}{\delta_\Gamma^2}\, \Big| \frac{\fy}{M} \Big|^2.
\end{equation*}
Since $\nfrac{1}{\delta_\Gamma}$ is bounded in $\Omega_\eps$ we easily deduce that
\begin{equation*}
\int_{\Omega_\eps} \frac{M}{\delta_\Gamma^2}\, \Big| \frac{\fy}{M} \Big|^2
\lesssim
\int_{\Omega_\eps} M \Big| \frac{\fy}{M} \Big|^2
\leq 
\|\fy\|_{L^2_M}^2
\end{equation*}
and the main difficulty to prove the Lemma~\ref{lem-hardy} is concentrated in the control of the integral 
\begin{equation*}
I_0 := \int_{\mathcal O_\eps} \frac{M}{\delta_\Gamma^2}\, \Big| \frac{\fy}{M} \Big|^2.
\end{equation*}
For similar reasons, we can suppose that
\begin{equation}\label{eq:fnulleaubord}
  \fy\big|_{\overline{\Omega_\eps}\cap \overline{\mathcal O_\eps}} = 0.
\end{equation}
In fact, introduce a regular function~$\gamma_\eps$ (only depending on~$\Omega$ and~$\eps$) wich vanished outside the $\eps$-neighborhood~$\mathcal O_\eps$ of~$\Gamma$ and is equal to~$1$ in the $\frac{\eps}{2}$-neighborhood. The estimate in Lemma~\ref{lem-hardy} clearly holds for $(1-\gamma_\eps)\fy$ and the proof therefore has to by conducted only for~$\gamma_\eps\fy$, which for sake of simplicity will be denote~$\fy$ in the following.\\
%
%----------------------------------------
\begin{figure}[htbp]
\begin{center}
{\psfrag{Gamma}{$\Gamma$}\psfrag{un}{$\gamma_\eps=1$}\psfrag{zero}{$\gamma_\eps=0$}\psfrag{Omegaeps}{$\Omega_\eps$}\psfrag{Oeps}{$\mathcal O_\eps$}\psfrag{bord}{$\overline{\Omega_\eps}\cap \overline{\mathcal O_\eps}$}
\includegraphics[width=4cm]{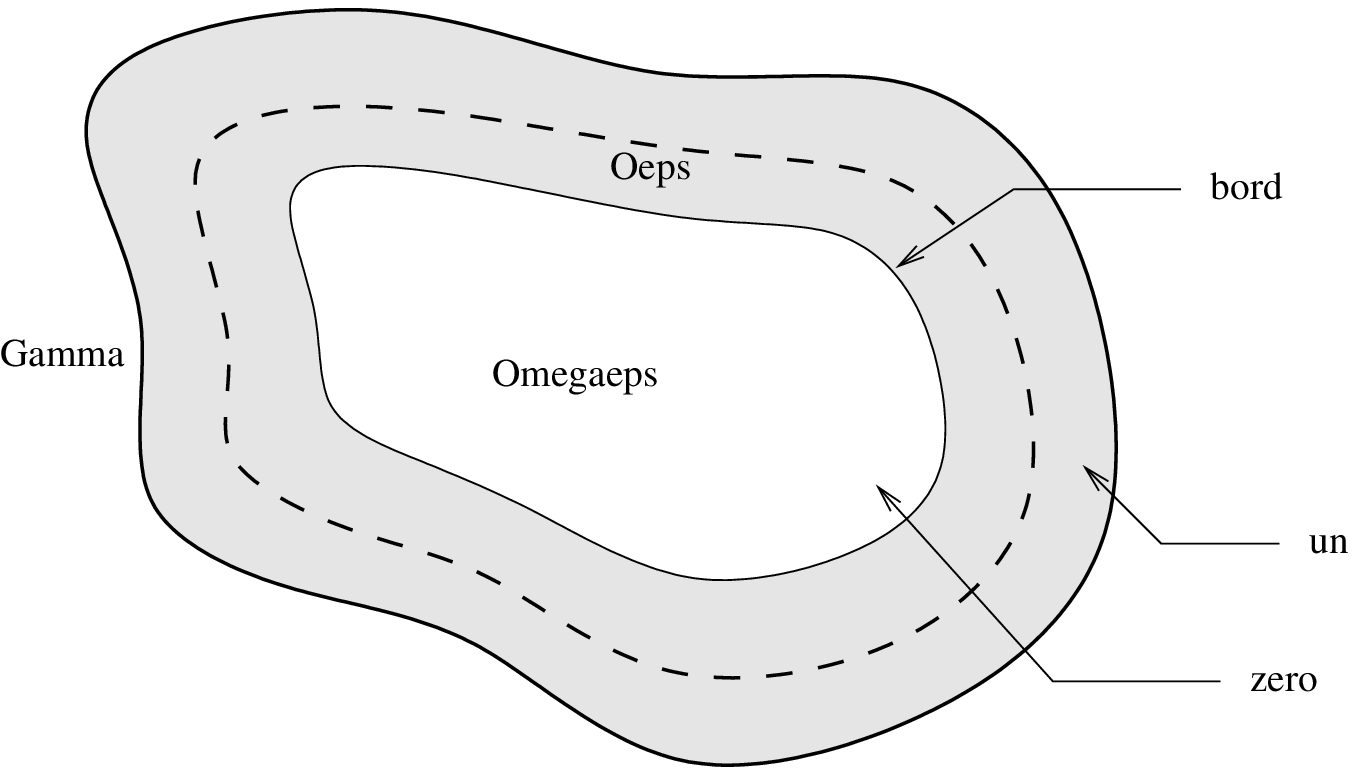}
}
%\caption{}\label{}
\end{center}
\end{figure}
%----------------------------------------
%

\noindent The value of $\eps$ being able to be chosen as small as desired, we shall use copiously the change of variables~\eqref{change_of_variables} as well as approximation~\eqref{change_of_variables_approx}.\\
%------------------------------
For each $\btheta\in \Gamma$ let introduce the function
\begin{equation*}
h_{\btheta}:r\in\, ]0,\eps] \mapsto \frac{ \fy(r,\btheta)}{\sqrt{ M(r,\btheta)}}.
\end{equation*}
Note that the relation~\eqref{eq:fnulleaubord} implies that~$h_{\btheta}(\eps)=0$.\\
To simplify the notations, when there are no ambiguities, we do not note the dependence with respect to the variable~$\btheta$ in~$h_{\btheta}$ and in~$M$. 
%In particular the operation denoted by $'$ corresponds to the derivation with respect to the variable~$r$. 
The integral~$I_0$ writes using the approximation~\eqref{change_of_variables_approx}
\begin{equation*}
I_0 
=
\int_{\Gamma} \int_0^\eps \frac{h(r)^2}{r^2}\, dr\, d\btheta
\end{equation*}
The goal is to control~$I_0$ with $\|\fy\|_{H^1_M}$. We will proceed in two steps:

\begin{enumerate}
\item We prove that $h(r)\nfrac{M'(r)}{M(r)}=0$ on $r=0$;
\item We control~$I_0$ with $\|\fy\|_{H^1_M}$.
\end{enumerate}
%
%-------------------------------
$\bullet$ Step~$(1)$: We use the following change of variable adapted to the maxwellian~$M$:
\begin{equation*}
\begin{aligned}
\Phi : \, (r,\btheta)\in \,]0,\eps]\times \Gamma \, \longmapsto & \, (s,\btheta)\in [0,+\infty[\times \Gamma,\\
& \text{where}\quad s=\int_r^\eps \frac{dt}{\widetilde M (t,\btheta)}.
\end{aligned}
\end{equation*}
Notice that the jacobian determinant of $\Phi(r,\btheta)$ egals $\nfrac{-1}{\widetilde M (r,\btheta)}$ and therefore, $M$ being positive on~$\Omega$, it does not cancel on $]0,\eps]\times \Gamma$. Moreover, for all $\btheta\in \Gamma$ we have $\lim_{r\rightarrow 0}\Phi(r,\btheta)=(+\infty,\btheta)$ since using the assumption on~$M$ we get $\int_0^\eps \nfrac{dr}{\widetilde M (r,\btheta)} = +\infty$. Consequently $\Phi$ is a local diffeomorphism from $]0,\eps]\times \Gamma$ to $\Phi(]0,\eps]\times \Gamma) = [0,+\infty[\times \Gamma$.\\
For any function $\widetilde f:\, ]0,\eps[\times \Gamma \mapsto \R$ we will define the function, noted $\widehat f:\, ]0,+\infty[\times \Gamma$, by $\widetilde f = \widehat f \circ \Phi$. Using the change of variables introduce in Subsection~\ref{subsection_geometry}, for any function $f:\mathcal O_{\eps} \mapsto \R$ we have define the functions $\widetilde f:\, ]0,\eps[\times \Gamma \mapsto \R$ and $\widehat f:\, ]0,+\infty[\times \Gamma$ such that, with the previous notations for the name of variables:
\begin{equation*}
f(\x) = \widetilde f(r,\theta) = \widehat f(s,\theta).
\end{equation*}
%---------------------
%
%dire aussi que l'on approche J par 1, ce qui est justifié par l'approximation et le fait qu'on peut prendre $\eps$ aussi petit que l'on veut.\\
%
In the nonambiguous cases, we will note~$f$, all the functions~$f$,~$\widetilde f$ or $\widehat f$.\\[0.3cm]
As announced before the proof, we only need the radial part of the gradient in the desired estimate. In term of new coordinates, the radial part of the gradient of a function~$f$ defined on~$\Omega_\eps$ corresponds to the derivative with respect to the variable~$r$ in the new coordinates: $\nabla_R f(\x) = \partial_r \widetilde f(r,\btheta)$.\\[0.3cm]
For any $\btheta \in \Gamma$ let $g_{\btheta}$ be the function defined on $]0,+\infty[$ by 
\begin{equation*}
g_{\btheta}(s) := \frac{\fy(\x)}{M(\x)} = \frac{h_{\btheta}(r)}{\sqrt{M(r,\btheta)}}.
\end{equation*}
Derivating with respect to the variable $s$, we obtain (as previously the variable~$\btheta$ will be understood as a parameter and we do not note its dependence): 
\begin{equation*}
M(\x) \left|\nabla_R \left( \frac{\fy(\x)}{M(\x)} \right) \right|^2 d\x 
=
- |g'(s)|^2 |J(r,\btheta)| \, d\btheta\, ds.
\end{equation*}
We deduce that (using the approximation $J\sim 1$ valid for $\eps$ small enough)
\begin{equation*}
I_1:= \int_\Omega M(\x)\left| \nabla_R \left( \frac{\fy(\x)}{M(\x)} \right) \right|^2 d\x 
=
\int_{\Gamma} \left( \int_0^{+\infty} |g'(s)|^2\, ds \right) \, d\btheta.
\end{equation*}
Since $\fy \in H^1_M$ the integral $I_1$ is finite. Consequently, for almost every $\btheta\in \Gamma$ we have 
\begin{equation*}
\int_0^{+\infty} |g'(s)|^2\, ds < +\infty.
\end{equation*}
We deduce that
\begin{equation*}
\forall \alpha > 0 \quad \exists s_\alpha >0 \quad \int_{s_\alpha}^{+\infty} |g'(s)|^2\, ds \leq \alpha^2.
\end{equation*}
Next, using the Cauchy-Schwarz inequality we obtain, for any $s\in\, ]0,+\infty[$,
\begin{equation*}
g(s) = g(s_\alpha)+\int_{s_\alpha}^s g'(t)\, dt 
\leq \left\{
\begin{aligned}
& g(s_\alpha)+\alpha \sqrt{s-s_\alpha} \quad \text{if $s>s_\alpha$}, \\
& \qquad \sup_{[0,s_\alpha]}g \hspace{1.4cm} \text{if $0<s\leq s_\alpha$}.
\end{aligned}
\right.
\end{equation*}
We deduce that  for almost every $\btheta\in \Gamma$, for all $\alpha > 0$ there exists a constant $C_{\btheta,\alpha}>0$ such that, for any $s\in\, ]0,+\infty[$ we have $g(s)\leq \alpha \sqrt{s}+C_{\btheta,\alpha}$. Using the $r$ variable, this result is written: for any $r\in\, ]0,\eps]$
\begin{equation*}
\frac{M'(r)}{M(r)}h(r)^2 \leq \alpha² M'(r) \int_r^\eps \frac{dt}{M(t)} + M'(r) C_{\btheta,\alpha}^2,
\end{equation*}
that enables, see assumption~\eqref{H1}, to obtain for almost every $\btheta \in \Gamma$ the relation
\begin{equation}\label{boundary_term}
\left.\frac{M'(r)}{M(r)}h(r)^2\right|_{r=0}=0.
\end{equation}
%
%----------------------------
$\bullet$ Step~$(2)$: Now, we prove the lemma. Since $\fy \in H^1_M$, we know that
\begin{equation*}
I_1 := \int_\Omega M(\x)\Big| \nabla_R \Big( \frac{\fy(\x)}{M(\x)} \Big) \Big|^2 d\x \, \leq \|\fy \|_{H^1_M}^2 < \, +\infty.
\end{equation*}
We express $I_1$ making appear the~$h$ function and using the change of variables $\x\mapsto (r,\btheta)$ together the usual approximation for the jacobian determinant of this change of variable (see the Part~\ref{subsection_geometry} and the relations~\eqref{change_of_variables_approx}).
We obtain
\begin{equation*}
I_1 
=
\int_{\Gamma} \int_0^\eps M(r,\btheta)\Big| \frac{\partial}{\partial r} \Big( \frac{\fy(r,\btheta)}{M(r,\btheta)} \Big) \Big|^2 \, dr\, d\btheta 
:=
\int_\Gamma I_1(\btheta)\, d\btheta,
\end{equation*}
where, for $\btheta\in \Gamma$, the quantity $I_1(\btheta)$ is defined by
\begin{equation*}
\begin{aligned}
I_1(\btheta)
% = \int_0^\eps M(r)\Big| \frac{\partial}{\partial r} \Big( \frac{ \fy(r,\btheta)}{M(r)} \Big) \Big|^2 dr 
& = \int_0^\eps M(r) \Big( \frac{1}{\sqrt{M(r)}}  h'(r) - \frac{M'(r)}{2M(r) \sqrt{M(r)}} h(r)\Big)^2 dr\\
& = \int_0^\eps \Big(  h'(r)^2 - \frac{M'(r)}{M(r)} h'(r) h(r) + \frac{M'(r)^2}{4M(r)^2} h(r)^2 \Big)\, dr.
\end{aligned}
\end{equation*}
Moreover, an integration by part gives
\begin{equation*}
- \int_0^\eps \frac{M'(r)}{M(r)}h'(r)h(r)\, dr
= \frac{1}{2} \int_0^\eps \Big( \frac{M'(r)}{M(r)} \Big)' h(r)^2\, dr
- \frac{1}{2} \Big[ \frac{M'(r)}{M(r)} h(r)^2 \Big]_0^\eps.
\end{equation*}
We deduce that
\begin{equation*}
\begin{aligned}
& I_1(\btheta)
= \int_0^\eps \Big( h'(r)^2 + \lambda(r) \frac{h(r)^2}{4 r^2} \Big)\, dr 
- \frac{1}{2} \Big[ \frac{M'(r)}{M(r)} h(r)^2 \Big]_0^\eps \\
& \quad \text{with} \quad 
\lambda(r) = \Big( \Big( \frac{M'(r)}{M(r)} \Big)^2 + 2 \Big( \frac{M'(r)}{M(r)} \Big)' \,\Big) r^2.
\end{aligned}
\end{equation*}
The assumption~\eqref{H1} on~$M$ is written $\lambda \geq -a >-1$. Moreover using the equation~\eqref{boundary_term} and the relation~\eqref{eq:fnulleaubord} the braket term is vanished. We obtain
\begin{equation*}
I_1(\btheta)
\geq
\int_0^\eps \Big( h'(r)^2 - a \frac{h(r)^2}{4 r^2} \Big)\, dr.
\end{equation*}
Moreover, thanks to the Hardy inequality (holds since~$h$ vanishes at~$0$, this is a direct consequence of equation~\eqref{boundary_term}), we deduce
\begin{equation}\label{estimate_derivate}
I_1(\btheta)
\geq
(1-a) \int_0^\eps h'(r)^2\, dr.
\end{equation}
Since $a<1$, this control allows us to estimate $\int_0^\eps h'(r)^2 dr$. From the Hardy inequality again, we obtain the following estimate
\begin{equation*}
\int_0^\eps \frac{h(r)^2}{r^2} dr \leq 4 \int_0^\eps h'(r)^2 dr
\leq
\frac{4}{1-a} I_1(\btheta).
\end{equation*}
Integrate with respect to the variable~$\btheta$, we obtain
\begin{equation*}
I_0 
=
\int_{\Gamma} \int_0^\eps \frac{h(r)^2}{r^2}\, dr\, d\btheta
\leq
\frac{4}{1-a} I_1.
\end{equation*}
The function $\fy$ being in $H^1_M$, right hand side termes are bounded by $\|\fy\|_{H^1_M}^2$. Up to the change of variables we deduce
\begin{equation*}
\int_{\mathcal O_\eps} \frac{M}{\delta_\Gamma^2} \, \Big | \frac{\fy}{M} \Big|^2 \lesssim \|\fy\|_{H^1_M}^2,
\end{equation*}
which implies the announced result.
\end{proof}
\noindent The additional hypothesis 
\begin{equation}\label{H2}\tag{$\mathcal H_2$}
\exists c>0 \qquad |\nabla M| \leq \frac{1}{c}\, \frac{M}{\delta_\Gamma}
\end{equation}
implies from Lemma~\ref{lem-hardy} that if $\fy\in H_M^1$ then $\frac{\nabla M}{M}\frac{\fy}{\sqrt{M}} \in L^2(\Omega)$. It is in this form that the Lemma~\ref{lem-hardy} will be generally used in this article.\\
Note that in term of~$\x$ variable, the inequality~\eqref{estimate_derivate} show us that $\nabla_R ( \frac{\fy}{\sqrt{M}} )$ belongs to~$L^2(\Omega)$. This propertie is completed by the next lemma:
%
%------------------------------------
\begin{lemma}[Inclusion]\label{lem-inclusion}
If~$M$ satisfies~\eqref{H1} and~\eqref{H2} then we have the following inclusions
\begin{equation*}
L^2_M\subset L^2(\Omega) \quad \text{and} \quad H^1_M\subset H^1_0(\Omega).
\end{equation*}
More precisely, if $\fy \in L^2_M$ then $\dsp \nfrac{\fy}{\sqrt{M}}\in L^2(\Omega)$ 
and if $\fy \in H^1_M$ then $\dsp \nfrac{\fy}{\sqrt{M}}\in H^1(\Omega)$.
\end{lemma}
%
%------------------------------------
\begin{proof} 
Inclusion $L^2_M\subset L^2(\Omega)$ is obvious since on the one hand, by definition of~$L^2_M$, we have $\fy\in L^2_M$ if and only if $\dsp \nfrac{\fy}{\sqrt{M}}\in L^2(\Omega)$, and on the other hand $\sqrt{M}\in L^\infty(\Omega)$.\\
\noindent To prove the inclusion $H^1_M\subset H^1_0(\Omega)$ we use the Lemma~\ref{lem-hardy} with additional assumption~\eqref{H2}: if $\fy\in H^1_M$ then we have
\begin{equation*}
\nabla \Big( \frac{\fy}{\sqrt{M}} \Big) 
= \nabla \Big( \frac{\fy}{M}\sqrt{M} \Big) 
=
\underbrace{\sqrt{M}\nabla \Big( \frac{\fy}{M} \Big)}_{L^2(\Omega)} 
+
\underbrace{\frac{1}{2}\frac{\nabla M}{M}\frac{\fy}{\sqrt{M}}}_{L^2(\Omega)} \in L^2(\Omega).
\end{equation*}
Consequently $\nfrac{\fy}{\sqrt{M}}\in H^1(\Omega)$. Hence this function $\nfrac{\fy}{\sqrt{M}}$ has a trace on the boundary $\Gamma$. Since $M$ is a regular function on $\Omega$ vanishing on $\Gamma$, we deduce that $\fy = \frac{\fy}{\sqrt{M}} \sqrt{M}\in H^1_0(\Omega)$.
\end{proof}
%
%------------------------------------
\noindent This next lemma is interesting in themselves for understanding the space~$H^1_M$ better. Moreover, it will be used in the proof of the Lemma~\ref{lem-test}.
%------------------------------------
\begin{lemma}[Density]\label{lem-density}
If~$M$ satisfies~\eqref{H1} and~\eqref{H2} then we have the following equality
\begin{equation*}
\overline{\mathcal C^\infty_0}^{H^1_M} = H^1_M.
\end{equation*}
\end{lemma}
%
%------------------------------------
\begin{proof}
Let $\fy\in H^1_M$ and define, for $n\in \N^*$, the function $\fy_n$ by
\begin{equation*}
\fy_n = \fy \, \chi \Big( \frac{1}{nM} \Big) \quad \text{where} \quad
\chi(t) 
= \left\{
\begin{aligned}
1 \qquad &\text{if $0\leq t<1$},\\
2-t \quad &\text{if $1\leq t< 2$},\\
0 \qquad &\text{if $t\geq 2$}.
\end{aligned}
\right.
\end{equation*}
We successively prove that 
\begin{enumerate}
\item the functions $\fy_n$ are in $H^1_M$,
\item we can approach these functions $\fy_n$ with $C^\infty_0$ functions,
\item the sequence $\{\fy_n\}_{n\in \N^*}$ converges to $\fy$ in $H^1_M$ sense.
\end{enumerate}
These three points clearly implicate the lemma.\\[0.2cm]
$\bullet$ $(1)$ By definition of $\chi$, we have for all $n\in \N^*$ the relation $|\fy_n|\leq |\fy|$ hence $\int_\Omega M\big| \frac{\fy_n}{M} \big|^2 \leq \int_\Omega M\big| \frac{\fy}{M} \big|^2$ which is bounded since $\fy\in H^1_M\subset L^2_M$. To control the gradient part of the $H^1_M$-norm of $\fy_n$, $n\in \N^*$, we write
\begin{equation*}
\begin{aligned}
\Big| \nabla \Big( \frac{\fy_n}{M} \Big) \Big| 
& = \Big| \nabla \Big( \frac{\fy}{M} \Big) \chi \Big( \frac{1}{nM} \Big) 
 + \frac{\fy}{M} \nabla \Big( \chi \Big( \frac{1}{nM} \Big)\Big) \Big|\\
& \leq \Big| \nabla \Big( \frac{\fy}{M} \Big) \Big| + \Big| \frac{\fy}{M} \nabla \Big( \chi \Big( \frac{1}{nM} \Big)\Big) \Big|.
\end{aligned}
\end{equation*}
By definition of the trucature function~$\chi$, the last term is not egal to $0$ if and only if $\frac{1}{2n} \leq M < \frac{1}{n}$. In this case it write $\big| \frac{1}{n} \frac{\fy}{M} \frac{\nabla M}{M^2} \big|$ and can be controlled by $2 \big| \frac{\fy}{M} \frac{\nabla M}{M} \big|$. We have
\begin{equation*}
\Big| \nabla \Big( \frac{\fy_n}{M} \Big) \Big|^2
\leq
\Big| \nabla \Big( \frac{\fy}{M} \Big) \Big|^2
+
4 \Big| \frac{\fy}{M} \frac{\nabla M}{M} \Big|^2.
\end{equation*}
Since $\fy\in H^1_M$ and using the Lemma~\ref{lem-hardy} together with the assumption~\eqref{H2}, we deduce that $\int_\Omega M \big| \nabla \big( \frac{\fy_n}{M} \big) \big|^2$ is bounded. Consequently, $\fy_n\in H^1_M$.\\[0.2cm]
$\bullet$ $(2)$ For each $n\in \N^*$ the function $\fy_n\in H^1_M\subset H^1$ is egal to~$0$ in a neighborhood of the boundary $\Gamma$.
Approaching $\fy_n$ by a sequence $\{\fy_{n,m}\}_{m\in\N}$ (which is egal to~$0$ on a neighborhood of~$\Gamma$) in~$H^1_0$ allows to approach~$\fy_n$ by the same sequence $\{\fy_{n,m}\}_{m\in\N}$ in~$H^1_M$.\\[0.2cm]
$\bullet$ $(3)$ Note that $\fy(\x)-\fy_n(\x)$ egals to $0$ if $M(\x)>\frac{1}{n}$. In the other case, that is for all $\x\in \Omega$ such that $M(\x)\leq \frac{1}{n}$, we get $|\fy_n(\x)-\fy(\x)|\leq |\fy(\x)|$ since $0\leq \chi\leq 1$. Hence
\begin{equation*}
\int_\Omega M\Big| \frac{\fy_n-\fy}{M} \Big|^2 \leq \int_{\x\in \Omega\,;\, M(\x)\leq \frac{1}{n}} M\Big| \frac{\fy}{M} \Big|^2.
\end{equation*}
Since $\fy\in H^1_M\subset L^2_M$ and since the measure of the set $\{\x\in \Omega\,;\, M(\x)\leq \frac{1}{n}\}$ tends to $0$ when $n$ tends to $+\infty$, we have
\begin{equation}\label{estim1:fyn-fy}
\int_\Omega M\Big| \frac{\fy_n-\fy}{M} \Big|^2 \; \xrightarrow{n\rightarrow +\infty} \; 0.
\end{equation}
Concerning the gradient part, we write
\begin{equation*}
\nabla \Big( \frac{\fy_n-\fy}{M} \Big) 
= \nabla \Big( \frac{\fy}{M} \Big) \Big( \chi\Big( \frac{1}{nM} \Big) - 1 \Big) 
+ \frac{\fy}{M} \nabla \Big( \chi\Big( \frac{1}{nM} \Big) - 1 \Big).
\end{equation*}
The first term is non zero if $M\leq \frac{1}{n}$ and is bounded by $\big| \nabla \big( \frac{\fy}{M} \big) \big|$ in this case. The last term is non zero if $\frac{1}{2n} \leq M < \frac{1}{n}$ and is bounded by $2\big| \frac{\nabla M}{M}\frac{\fy}{M} \big|$. Hence we have
\begin{equation*}
\int_\Omega M\Big|\nabla \Big( \frac{\fy_n-\fy}{M} \Big) \Big|^2
\leq
\int_{\x\in \Omega\,;\, M(\x)\leq \frac{1}{n}} M\Big|\nabla \Big( \frac{\fy}{M} \Big) \Big|^2  
+
4 \Big| \frac{\nabla M}{M}\frac{\fy}{\sqrt{M}} \Big|^2.
\end{equation*}
Since $\fy\in H^1_M$, the first term is bounded, and using the result of the Lemma~\ref{lem-hardy} again, we know that the last term is bounded too. More precisely, as previously, using the fact that the measure of the set $\{\x\in \Omega\,;\, M(\x)\leq \frac{1}{n}\}$ tends to $0$ when $n$ tends to $+\infty$, we have
\begin{equation}\label{estim2:fyn-fy}
\int_\Omega M\Big|\nabla \Big( \frac{\fy_n-\fy}{M} \Big) \Big|^2 \; \xrightarrow{n\rightarrow +\infty} \; 0.
\end{equation}
The relations $\eqref{estim1:fyn-fy}$ and $\eqref{estim2:fyn-fy}$ ensure that the sequence $\{\fy_n\}_{n\in \N^*}$ converges to $\fy$ in $H^1_M$.
\end{proof}
%
%------------------------------------
\noindent Now we introduce a compacity result for the spaces $L^2_M$ and $H^1_M$ which is comparable to the classical compact injection $H^1_0(\Omega)\hookrightarrow L^2(\Omega)$. 
%
%-------------------------------------
\begin{lemma}[Compacity]\label{lem-compacity}
The injection $H^1_M\hookrightarrow L^2_M$ is compact.
\end{lemma}
\begin{proof} To prove this lemma, we use the following result due to G. Metivier~\cite[Proposition~3.1 p.~221]{Metivier} affirming that the weight Sobolev space injection $H^1(M dx)\subset L^2(M dx)$ as soon as we have~$M>0$ on~$\Omega$, $M=0$ on~$\Gamma$ and $\nabla M \neq 0$ out of~$\Gamma$. Notice that this last point is a consequence of the two first points~$M>0$ on~$\Omega$ and $M=0$ on~$\Gamma$, at least in a neighborooh of the boundary $\Gamma$, what is sufficient for our results. For sake of simplicity, we will nevertheless work with~$\Omega$.\\% It comes from to the following remark: we have $M=0$ on~$\Gamma$ and $M>0$ on~$\Omega$ then $\nabla_R M>0$ and $\nabla_{\btheta} M=0$ on $\gamma$, that implies there exists a neighborooh of the boundary $\Gamma$ in which $\nabla M\neq 0$.\\
Consider a sequence $\{\fy_n\}_{n\in \N}$ bounded in $H^1_M$ and show that a convergent sub-sequence can be extracted. By definition of $H^1_M$, for all $n\in \N$, there exists $g_n\in H^1(M dx)$ such that $\fy_n=Mg_n$. The sequence $\{\fy_n\}_{n\in\N}$ being bounded in $H^1_M$, the sequence $\{g_n\}_{n\in \N}$ is bounded in $L^2_M$. We use here the result of G. Métivier. We can extract from the sequence $\{g_n\}_{n\in\N}$ a sub-sequence, still noted $\{g_n\}_{n\in\N}$ and such that
\begin{equation*}
g_n \rightharpoonup g \quad \text{in $H^1(M dx)$} \qquad \text{and}\qquad g_n \rightarrow g \quad \text{in $L^2(M dx)$}.
\end{equation*}
By definition of the spaces $H^1_M$ and $L^2_M$ we conclude that
\begin{equation*}
\fy_n \rightharpoonup Mg \quad \text{in $H^1_M$} \qquad \text{and}\qquad \fy_n \rightarrow Mg \quad \text{in $L^2_M$,}
\end{equation*}
which proves that the injection $H^1_M \hookrightarrow L^2_M$ is compact.
\end{proof}
%-------------------------------------
\noindent In the same way, we present the next lemma which proves that functions in~$H^1_M$ are in certain~$L^p_M$, $p>2$, where the weighted-space~$L^p_M$ is defined by
\begin{equation*}
L^p_M = \Big\{ \fy\in  L^1_{loc}(\Omega) \ ;\ \Big( \int_\Omega M\Big|\frac{\fy}{M} \Big|^{p}\, \Big)^{1/p} < +\infty \Big\}
\end{equation*}
and endowed with its usual norm.\\
This kind of result is essential for the proof of the main theorem (Theorem~\ref{the-main}); it is proved under the assumptions~\eqref{H1} and~\eqref{H2}. We will note that assumption~\eqref{H2} is used in the following weak formulation (obtained by integration):
\begin{equation*}
  \eqref{H2} \quad \Longrightarrow \quad \exists C'>0 \quad  M\geq C'\, \delta_\Gamma^{\nfrac{1}{c}}.
\end{equation*}
More exactly, we have
%-------------------------------------
\begin{lemma}[Sobolev-type injection]\label{lem-injection}
If~$M$ satisfies~\eqref{H1} and~\eqref{H2} then there exists $p>2$ such that the injection $H^1_M \hookrightarrow L^p_M$ is continuous.
\end{lemma}
\begin{proof}
First, let us note that $\fy \in L^p_M$ if and only if $\dsp \frac{\fy}{M^{1-1/p}} \in L^p(\Omega)$ where the spaces $L^p(\Omega)$ are the classical Sobolev spaces on the set~$\Omega$. Let $\fy\in H^1_M$. In the next three steps we will prove that there exists $p>2$ such that $\dsp \frac{\fy}{M^{1-1/p}} \in L^p(\Omega)$.\\
$\bullet$ Since $\fy\in H^1_M$ we clearly have $\fy\in L^2_M$ and so $\nfrac{\fy}{\sqrt{M}}\in L^2$. Moreover, using assumptions Lemma~\ref{lem-hardy} with assumption~\eqref{H2} we obtain
\begin{equation*}
  \nabla \left( \frac{\fy}{\sqrt{M}} \right) 
= 
\underbrace{\frac{\nabla \fy}{\sqrt{M}}}_{L^2(\Omega)}
+
\underbrace{\frac{\nabla M}{M}\frac{\fy}{\sqrt{M}}}_{L^2(\Omega)}
\in L^2(\Omega).
\end{equation*}
Consequently, $\dsp \nfrac{\fy}{\sqrt{M}}\in H^1(\Omega)$ and using the classical Sobolev injections, we deduce that $\dsp \nfrac{\fy}{\sqrt{M}}\in L^q(\Omega)$ for all $\dsp q\leq 2d/(d-2)$ (and for $q< +\infty$ in the $2$-dimensional case).\\
$\bullet$ Using the assumption~\eqref{H2} and the Hardy inequality (Lemma~\ref{lem-hardy}) we obtain
\begin{equation*}
\frac{\fy}{\sqrt{M}M^c} = \underbrace{ \frac{1}{\delta_\Gamma} \frac{\fy}{\sqrt{M}} }_{L^2(\Omega)} \times \underbrace{ \frac{\delta_\Gamma}{M^c} }_{L^\infty(\Omega)} \in L^2(\Omega).
\end{equation*}
$\bullet$ From the two previous steps, we can write that for any $\beta\in \R$ such that $0\leq \beta\leq c$ we obtain
\begin{equation*}
\frac{\fy}{\sqrt{M}M^\beta} 
=
\underbrace{ \Big( \frac{\fy}{\sqrt{M}M^c} \Big)^{\beta/c} }_{L^{2c/\beta}(\Omega)} \times \underbrace{ \Big( \frac{\fy}{\sqrt{M}} \Big)^{1-\beta/c} }_{L^{qc/(c-\beta)}(\Omega)} \in L^r(\Omega)
\end{equation*}
with $r = \frac{2c q}{q\beta + 2(c-\beta)}$. Let us note $p$ the real such that $1-\frac{1}{p} = \beta+\frac{1}{2}$. The previous result is written: for any $p\in \R$ such that $2\leq p \leq \frac{2}{1-2c}$ we obtain
\begin{equation*}
\frac{\fy}{M^{1-1/p}} \in L^r
\qquad
\text{with \quad $\dsp r = \frac{4c p q}{qp - 2q + 4c p - 2p + 4}$.}
\end{equation*}
In particular, we have $\dsp \frac{\fy}{M^{1-1/p}} \in L^p$ as soon as $r\geq p$. The inequality $r\geq p$ holds if and only if $p\leq 2 + \frac{4c (q-2)}{4c + q - 2}$. 
It is thus possible to find $p>2$ such that $\dsp \frac{\fy}{M^{1-1/p}} \in L^p$.
\end{proof}
%------------------------------------
\begin{rema}
According to the previous proof, we obtain the inclusion $H^1_M\subset L^p_M$ for all $p\leq 2+\frac{4c}{c(d-2)+1}$. For instance, in the two dimensional case ($d=2$) the inclusion $H^1_M\subset L^p_M$ holds for all $p\leq 2+4c$.
\end{rema}
%------------------------------------
\noindent To build functions in $H^1_{M,0}$ we use the following lemma which will be important to obtain a lot of test functions in the weak formulation later. The hypothesis~\eqref{H1} allows to use the Lemma~\ref{lem-density}.
\begin{lemma}\label{lem-test}
Assume~\eqref{H1},~\eqref{H2} hold. Let $\psi\in H^1_M$ and $\xi:\R\rightarrow \R$ be a continuous application, piecewise-$\mathcal C^1$ such that~$\xi'$ is bounded on~$\R$. Then we have 
\begin{equation*}
\fy := M\xi \Big( \frac{\psi}{M}\Big) - M \int_\Omega M\xi \Big( \frac{\psi}{M}\Big) \in H^1_{M,0},
\end{equation*}
with $\dsp \nabla \Big( \frac{\fy}{M}\Big) = \xi'\Big( \frac{\psi}{M}\Big)\nabla \Big( \frac{\psi}{M}\Big)$ and $\dsp \|\fy\|_{H^1_{M,0}} \leq \|\xi'\|_{L^\infty(\R)} \|\psi\|_{H^1_{M,0}}$.
\end{lemma}
%
%-------------------------------------
\begin{proof}
The proof of this lemma uses the Stampacchia lemma which affirms that if $g\in H^1(\omega)$, $\omega$ being an open subset of~$\R^d$, and $\xi:\R\rightarrow \R$ is continuous, piecewise-$\mathcal C^1$, such that~$\xi'$ is bounded on~$\R$ then we have $\xi(g) \in H^1(\omega)$ and $\nabla\xi(g)= \xi'(g)\nabla g$.\par
\noindent The Stampacchia lemma is a local result, hence applied to $g=\nfrac{\psi}{M}$ which is in $H^1_{\text{loc}}(\Omega)$ it shows that the formula $\nabla \xi \big( \frac{\psi}{M}\big) = \xi'\big( \frac{\psi}{M}\big)\nabla \big( \frac{\psi}{M}\big)$ holds in~$\mathcal D'(\Omega)$.
From this formula, it is obvious that $\fy \in H^1_M$.
The fact that~$\fy$ is null average is then immediate since $\int_\Omega M=1$.
\end{proof}
%
%-------------------------------------
\noindent One more important ingredient in our study is the following linear operator
\begin{equation*}
\mathcal L \psi = -\div \Big( M \nabla \Big( \frac{\psi}{M} \Big) \Big)
\end{equation*}
on the space~$L^2_M$ and with domain, see~\cite[Remark~3.8, p.~9]{Masmoudi} given by
\begin{equation*}
D(\mathcal L) = \{\psi\in H^1_M\ ;\ \int_\Omega \frac{1}{M} \Big| \div \Big( M \nabla \Big( \frac{\psi}{M} \Big) \Big) \Big|^2 < +\infty\}.
\end{equation*}
We also find in~\cite[Proposition 3.6, p.~8]{Masmoudi} 
%\footnote{Notice that its proof is given for pecular maxwellian $M$ but can clearly be adapted to the present case. We also note the Remark 3.8 of this paper~\cite{Masmoudi} using the fact, as the Lemma~\ref{lem-inclusion} of the present article, that the boundary condition is consequence of the definition of the spaces~$L^2_M$ and~$H^1_M$.}
the following result and its proof which will be used to introduce the Galerkin approximation method later.
\begin{lemma}\label{Galerkin}
The operator $\mathcal L$ is self-adjoint and positive. Moreover, it has a discrete spectrum formed by a sequence $(\ell_n)_{n\in \N}$ such that~$\ell_n$ tends to $+\infty$ when~$n$ tends to $+\infty$.
\end{lemma}
%--------------------------------------
\noindent Concerning the uniqueness results for a linear operator, it is known that the eigenvalue~$0$, that is the kernel of the operator~$\mathcal L$, is particularly important.
\begin{lemma}\label{lem-kernel}
The kernel of the operator~$\mathcal L$ is the set $\{ \lambda M, \lambda \in \R\}$.
\end{lemma}
\begin{proof}
This lemma is an immediate consequence of the following formulation of the operator~$\mathcal L$:
\begin{equation}\label{operator-weak}
\langle \mathcal L \psi,\fy \rangle_{L^2_M} = \int_\Omega M\nabla\Big( \frac{\psi}{M} \Big) \cdot \nabla\Big( \frac{\fy}{M} \Big)
\end{equation}
where $\langle \cdot,\cdot \rangle_{L^2_M}$ corresponds to the scalar product subordinated to the norm $\|\cdot\|_{L^2_M}$ on~$L^2_M$. In fact, let~$\psi$ be a function such that $\mathcal L \psi=0$. We obtain $\langle \mathcal L \psi,\psi \rangle_{L^2_M} = 0$ and the formulation~\eqref{operator-weak} yields $\nabla \big( \nfrac{\psi}{M} \big) = 0$. Thus, thanks to the connexity of~$\Omega$, we deduce that $\psi = \lambda M$ with $\lambda \in \R$.
\end{proof}
%--------------------------------------
%\noindent In particular, since $\int_\Omega M = 1$, the kernel of $\mathcal L\big|_{H^1_{M,0}}$ is the null space~$\{0\}$.
%
%--------------------------------------
\noindent The last lemma is a generalized Poincaré inequality adapted to the weighted spaces introduced before. To obtain such a lemma, we use the fact that the potential $V=\ln(M)$ is concave. More precisely we will suppose that
\begin{equation}\label{H3}\tag{$\mathcal H_3$}
\exists \gamma>0 \quad \nabla \Big( \frac{\nabla M}{M} \Big) \leq -\gamma~\mathrm{Id}.
\end{equation}
\begin{lemma}[Poincaré-type inequality]\label{lem-poincare}
If~$M$ satisfies~\eqref{H3} then for all $\fy\in H^1_M$ we get the following Poincaré-type inequality
\begin{equation*}
\frac{1}{\gamma}\int_\Omega M \Big|\nabla \Big( \frac{\fy}{M} \Big) \Big|^2 + \Big( \int_\Omega \fy \Big)^2 \geq \|\fy\|_{L^2_M}^2.
\end{equation*}
\end{lemma}
\noindent For the free-average functions (that is for $\psi\in H^1_{M,0}$) this Lemma~\ref{lem-poincare} show that the two norms $\|\cdot\|_{H^1_M}$ and $\|\cdot\|_{H^1_{M,0}}$ on this space are equivalents. This equivalence will be usually useful in the remainder of the paper.
\begin{proof}
Let $\fy\in H^1_M$ and introduce the non-stationary problem
  \begin{equation*}
\left\{
\begin{aligned}
&    u_t(t,\x) + \mathcal L u(t,\x) = 0 \quad \text{for $(t,\x)\in \R\times \Omega$,}\\
&    u(0,\x)=\fy(\x) \hspace{1.55cm} \text{for $\x\in \Omega$.}
\end{aligned}
\right.
  \end{equation*}
The following time-dependant functions
\begin{equation*}
  D(t) = \int_\Omega M\Big| \frac{u}{M} \Big|^2
\quad \text{and} \quad
  H(t) = \int_\Omega M\Big| \nabla \Big( \frac{u}{M} \Big) \Big|^2
\end{equation*}
satisfy
\begin{equation*}
  D'(t)
=
  2 \int_\Omega \frac{u}{M}\cdot u_t
=
  -2 \int_\Omega \frac{u}{M}\cdot \mathcal Lu
=
  -2 \int_\Omega M\Big| \nabla \Big( \frac{u}{M} \Big) \Big|^2
=
  -2 H(t).
\end{equation*}
Moreover we have
\begin{equation*}
  H'(t)
=
  -2 \int_\Omega M \nabla \Big( \frac{u}{M} \Big) \cdot \nabla \Big( \frac{\mathcal Lu}{M} \Big).
\end{equation*}
For clearify the following computations, let us introduce the duality operator $\mathcal L^\star$ such that
\begin{equation*}
  \mathcal L^\star v := -\frac{1}{M} \div (M \nabla v) = -\Delta v - \frac{\nabla M}{M}\cdot \nabla v.
\end{equation*}
We have $\mathcal Lu = M\mathcal L^\star \big( \frac{u}{M} \big)$ and
\begin{equation*}
  H'(t)
=
  -2 \int_\Omega M \nabla \Big( \frac{u}{M} \Big) \cdot \nabla \Big( \mathcal L^\star \Big( \frac{u}{M} \Big) \Big).
\end{equation*}
Since $\nabla (\mathcal L^\star v) = \mathcal L^\star (\nabla v) - \nabla \big( \frac{\nabla M}{M} \big)\cdot \nabla v$ we deduce that
\begin{equation*}
\begin{aligned}
  H'(t)
& =
  - 2 \int_\Omega M \nabla \Big( \frac{u}{M} \Big) \cdot \mathcal L^\star \Big( \nabla \Big( \frac{u}{M} \Big) \Big) \\
& \hspace{2cm} + 2 \int_\Omega M \nabla \Big( \frac{u}{M} \Big) \cdot \Big[ \nabla \Big( \frac{\nabla M}{M} \Big) \cdot \nabla \Big( \frac{u}{M} \Big) \Big].
\end{aligned}
\end{equation*}
The first term of the right hand side is written $- 2 \int_\Omega M \big| \nabla^2 \big( \frac{u}{M} \big)\big|^2$ and is non-positive. Using the assumption~\eqref{H3}, the last term of the right hand side is controled by $- 2 \gamma \int_\Omega M \big| \nabla \big( \frac{u}{M} \big)\big|^2$, that is by $-2\gamma H(t)$. We obtain
\begin{equation*}
  H'(t) \leq -2 \gamma ~ H(t).
\end{equation*}
Hence $H(t) \leq H(0)e^{-2\gamma t}$. Integrate in time, we obtain:
\begin{equation}\label{1642}
  D(0) - D(+\infty) = 2 \int_0^{+\infty} H(t)~dt \leq \frac{1}{\gamma}H(0).
\end{equation}
To evaluate $D(+\infty)$, we consider a stationary solution~$u_\infty$. We note that due to the spectral properties of the operator~$\mathcal L$ (see Lemma~\ref{Galerkin}), for any initial data,~$u$ tends to a stationary solution~$u_\infty$ as $t\rightarrow +\infty$. By definition it is in the kernel of $\mathcal L$ and following the Lemma~\ref{lem-kernel} there exists a constant~$\lambda$ such that $u_\infty = \lambda M$. But the evolution equation on~$u$ implies that the mean value $\int_\Omega u$ is conserved: $\int_\Omega u_\infty = \int_\Omega \fy$, that allows to obtain the constant $\lambda=\int_\Omega \fy$. We deduce that
\begin{equation*}
  D(+\infty) = \int_\Omega M\Big| \frac{u_\infty}{M} \Big|^2 = \Big( \int_\Omega \fy \Big)^2.
\end{equation*}
Consequently, the inequality~\eqref{1642} corresponds to the following one
\begin{equation*}
\int_\Omega  M \Big| \frac{\fy}{M} \Big|^2
-
\Big( \int_\Omega \fy \Big)^2
\leq
\frac{1}{\gamma} \int_\Omega  M \Big|\nabla \Big( \frac{\fy}{M} \Big) \Big|^2
\end{equation*}
which exactly is the inequality announced by the Lemma~\ref{lem-poincare}.
\end{proof}
%----------------------------------
\noindent Notice that it is possible to obtain a proof of this Poincaré-type inequality by contradiction, see for instance \cite[p.7]{Masmoudi}, or peraphs using the hole-space case (for example for $\Omega=\R^d$) proved in H.J.~Brascamp~\cite{Brascamp} (see also Proposition 2.1 in~\cite{Degond}).

%%%%%%%%%%%%%%%%%%%%%%%%%%%%%%%%%%%%%%%%%%%%%%%%%%%%%%%%%%%%%%%%%%%%%%%%%%%%%%
%%%%%%%%%%%%%%%%%%%%%%%%%%%%%%%%%%%%%%%%%%%%%%%%%%%%%%%%%%%%%%%%%%%%%%%%%%%%%%
\section{Statement of the main theorem}\label{sec-assumptions}
%%%%%%%%%%%%%%%%%%%%%%%%%%%%%%%%%%%%%%%%%%%%%%%%%%%%%%%%%%%%%%%%%%%%%%%%%%%%%%
%%%%%%%%%%%%%%%%%%%%%%%%%%%%%%%%%%%%%%%%%%%%%%%%%%%%%%%%%%%%%%%%%%%%%%%%%%%%%%

\subsection{Definition of weak solution}\label{sub1-1}
%%%%%%%%%%%%%%%%%%%%%%%%%%%%%%%%%%%%%%%%%%%%%%%%%%%%%%%%%%%%%%%%%%%%%%%%%%%%%%

\noindent When we consider the Fokker-Planck equation~\eqref{FP} with vector field~$\F$ decomposed as the sum $\F=\bkappa+\nabla V$ where $\bkappa\in L^\infty(\Omega)$ and $e^V\in L^1(\Omega)$, we can introduce the maxwellian function~$M$ by $M = \frac{e^V}{\int_\Omega e^V}$ and rewrite~\eqref{FP} as $\div ( \fy \bkappa - M\nabla \big( \frac{\fy}{M} \big)) = f$.
If we look for a solution with given average, that is for instance a solution such that $\int_\Omega \fy=1$, then we can reduce to the case where~$\fy$ is free-average exchanging~$\fy$ into $\fy-M$ and~$f$ into $f-\div(M\, \bkappa)$. We obtain the following problem
\begin{equation*}\label{pb1-fort}
\left\{
\begin{aligned}
\div ( \fy \bkappa - M\nabla \Big( \frac{\fy}{M} \Big)) &= f \qquad \text{in $\Omega$},\\
\text{with}\quad \int_\Omega \fy &= 0.
\end{aligned}
\right.
\end{equation*}
Using the adapted spaces introduce in the previous part, the weak formulation of this equation is written: find $\fy\in H^1_{M,0}$ such that for all $\psi\in H^1_{M,0}$
\begin{equation}\label{pb1-faible}
\int_\Omega M \nabla\left( \frac{\fy}{M}\right) \cdot \nabla\left( \frac{\psi}{M}\right) 
-
\int_\Omega \fy \bkappa \cdot \nabla\left( \frac{\psi}{M}\right) 
=
\langle f,\psi \rangle
\end{equation}
where $\langle \cdot,\cdot \rangle$ denote the duality brackets between $H^{-1}_M$ and $H^1_{M,0}$. 
\subsection{Assumptions on the potential}\label{sub1-2}
%%%%%%%%%%%%%%%%%%%%%%%%%%%%%%%%%%%%%%%%%%%%%%%%%%%%%%%%%%%%%%%%%%%%%%%%%%%%%%
%
\noindent In this article we are interested in the case where the vector fields $\F$ quickly explodes near to the boundary. The fact that $\F$ is decomposed as a sum of two terms makes it possible to describe all the ``explosive'' behavior in the part $\nabla V$. 
%As we will see in the simple example in the case $n=1$ (see the Subsection~\ref{sub1-3}, the cancellation of the function~$M$ on the boundary does not suffice and results depends on the behavior of~$M$ near to the boundary.
In addition to the fact that~$V$ equals~$-\infty$ on~$\Gamma$ to ensure the explosion, the assumptions given on~$V$ (or on~$M$, which is equivalent) can be checked only in a neigborhood of the boundary~$\Gamma$.
More precisly in order to use the lemmas proved we will use the following assumptions
\begin{equation}\label{H1}\tag{$\mathcal H_1$}
\left\{
\begin{aligned}
& \exists a<1 \quad \Big( \frac{\nabla_RM}{M} \Big)^2 + 2 \, \nabla_R\Big( \frac{\nabla_RM}{M} \Big)  \geq \frac{-a}{\delta_\Gamma^2}, \\
& \nabla_R M(0)=0, \\
& \exists b>0 \quad \nabla_R M \int_\Omega \frac{1}{M} < b,
\end{aligned}
\right.
\end{equation}
\begin{equation}\label{H2}\tag{$\mathcal H_2$}
\exists c>0 \qquad |\nabla M| \leq \frac{1}{c}\, \frac{M}{\delta_\Gamma},
\end{equation}
\begin{equation}\label{H3}\tag{$\mathcal H_3$}
\exists \gamma>0 \quad \nabla \Big( \frac{\nabla M}{M} \Big) \leq -\gamma~\mathrm{Id},
\end{equation}
where we recall that~$\nabla_R$ corresponds to the normal derivative and where~$\delta_\Gamma$ represents the distance to~$\Gamma$.\\
Notice that we can rewrite these assumptions in term of the potential $V$ (wich is given with respect to the maxwellian $M$ by $V=\ln M$), see for instance Theorem~\ref{the-main}, page~\pageref{the-main}.
It is important to note that these assumptions are satisfied for the radial functions~$M$ (i.e. functions depending only on the distance to the boundary) on the following form near to the boundary
\begin{equation*}
M(r) = r^\alpha \quad \text{with $\alpha>1$}.
\end{equation*}
In other words, the result is shown for vector fields~$\F$ whose the normal component explodes like~$\frac{\alpha}{\delta_\Gamma}$ with $\alpha>1$.
%This result can appear surprising but it corresponds to the result obtained in the examples given page~\pageref{sub1-3}. It is probable that the result (of uniqueness) is false when $\F$ explode like~$\frac{\alpha}{\delta_\Gamma}$ with $\alpha\leq 1$.
%
\begin{rema}
As it was announced as introduction, an interesting case corresponds to the following Fokker-Planck equation 
\begin{equation*}
-\eps \Delta \fy + \div (\fy \widetilde \F) = f,
\end{equation*}
making appear a small parameter~$\eps$. We can come back to the previous case using $\F= \frac{1}{\eps} \widetilde \F$. We note that if we define a Maxwellian~$M$ such that $\F= \nfrac{\nabla M}{M}$ then the Maxwellian $\widetilde M$ adapted to $\widetilde \F$, i.e. such that $\widetilde \F=\nfrac{\nabla \widetilde M}{\widetilde M}$, satisfies $M=C\,{\widetilde M}^{1/\eps}$. The assumptions on~$M$ can thus be interpreted on~$\widetilde M$ and we show that they are less constraining in the following sense: they are checked when the normal component of~$\widetilde \F$ behave like~$\frac{\alpha}{\delta_\Gamma}$ for all $\alpha>\eps$.
\end{rema}
\label{rema-regul_sur_kappa}
\noindent Concerning the assumption on the ``interior'' part~$\bkappa$ of $\F=\bkappa+\nabla V$, that is about $\bkappa \in L^\infty(\Omega)$, we can note that this assumption is stronger than that announced by J.~Droniou in~\cite{Droniou}.
In fact, we will see during the proof that the regularity required on~$\bkappa$ comes from a product lemma. Roughly speaking, if the product of a function~$H^1(\Omega)$ by a function~$L^p(\Omega) $ is a function~$L^2(\Omega)$ then the theorem is true as soon as~$\bkappa$ belongs to~$L^p(\Omega)$.
In the classical case the usual Sobolev injections $H^1(\Omega) \subset L^{2d/(d-2)}$ imply that $p=d_*$ is sufficient.
In our case the injections of ``Sobolev'' type (see the Lemma~\ref{lem-injection}) are not also ``generous'' and a product $H^1_M \times L^p(\Omega)$ will not belong to~$L^2(\Omega)$ for as many values of~$p$. We can possibly improve the result of the theorem by taking $\bkappa \in L^p(\Omega)$ with $p\geq d+\nfrac{1}{c}$.

\subsection{Main theorem}\label{sub1-toto}
%%%%%%%%%%%%%%%%%%%%%%%%%%%%%%%%%%%%%%%%%%%%%%%%%%%%%%%%%%%%%%%%%%%%%%%%%%%%%%

\noindent We prove in Part~\ref{sec-proof} the following theorem.
\begin{theorem}\label{th1}
Let $\Omega$ be a bounded domain of $\R^d$, $d\geq 2$. We denote by~$\Gamma$ its boundary which is assumed to be of class~$\mathcal C^2$.
Let $f\in H^{-1}_M$ and $\F=\bkappa+\nabla V$ where $\bkappa\in L^\infty(\Omega)$ and $V\in \mathcal C^\infty(\Omega)$ satisfies $V = -\infty$ on~$\Gamma$.\\
If the assumptions~\eqref{H1},~\eqref{H2} and~\eqref{H3} hold then the problem~\eqref{pb1-faible} admits a unique solution $\fy\in H^1_{M,0}$.
\end{theorem}
\noindent We can deduce - see the link between a free-average solution and a solution with given average on Subsection~\ref{sub1-1} - the following theorem where we recall all the assumptions
\begin{theo}\label{the-main}
Let $\Omega$ be a bounded domain of $\R^d$, $d\geq 2$. We denote by~$\Gamma$ its boundary which is assumed to be of class $\mathcal C^2$.
Let $f\in H^{-1}_M$ and $\F=\bkappa+\nabla V$ where $\bkappa\in L^\infty(\Omega)$ and $V\in \mathcal C^\infty(\Omega)$ satisfies $V = -\infty$ on $\Gamma$.\\
If we assume that, in a neigborhood of the boundary~$\Gamma$, we have
\begin{equation}\label{H1}\tag{$\mathcal H_1$}
\left\{
\begin{aligned}
& \exists \, a<1 \quad \Big( \nabla_RV \Big)^2 + 2 \, \nabla_R^2 V  \geq \frac{-a}{\delta_\Gamma^2}, \\
& \nabla_R V ~e^V=0 \quad \text{on $\Gamma$}, \\
& \exists \, b>0 \quad \nabla_R V ~e^V \int_\Omega e^{-V} < b,
\end{aligned}
\right.
\end{equation}
\begin{equation}\label{H2}\tag{$\mathcal H_2$}
\exists \, c>0 \qquad |\nabla V| \leq \frac{c}{\delta_\Gamma},
\end{equation}
\begin{equation}\label{H3}\tag{$\mathcal H_3$}
\exists \gamma>0 \quad \nabla^2 V \leq -\gamma~\mathrm{Id},
\end{equation}
where~$\nabla_R$ corresponds to the normal derivative and where~$\delta_\Gamma$ represents the distance to~$\Gamma$, then there exists a unique (weak) solution of the Fokker-Planck equation
\begin{equation*}
-\Delta \fy + \div(\fy \, \F) = f \quad \text{in $\Omega$},
\end{equation*}
such that $\int_\Omega \fy = 1$.
\end{theo}

%%%%%%%%%%%%%%%%%%%%%%%%%%%%%%%%%%%%%%%%%%%%%%%%%%%%%%%%%%%%%%%%%%%%%%%%%%%%%%
%%%%%%%%%%%%%%%%%%%%%%%%%%%%%%%%%%%%%%%%%%%%%%%%%%%%%%%%%%%%%%%%%%%%%%%%%%%%%%
\section{Proof of the Theorem~\ref{th1}}\label{sec-proof}
%%%%%%%%%%%%%%%%%%%%%%%%%%%%%%%%%%%%%%%%%%%%%%%%%%%%%%%%%%%%%%%%%%%%%%%%%%%%%%
%%%%%%%%%%%%%%%%%%%%%%%%%%%%%%%%%%%%%%%%%%%%%%%%%%%%%%%%%%%%%%%%%%%%%%%%%%%%%%

\subsection{Existence proof in Theorem~\ref{th1}}
%%%%%%%%%%%%%%%%%%%%%%%%%%%%%%%%%%%%%%%%%%%%%%%%%%%%%%%%%%%%%%%%%%%%%%%%%%%%%%

\noindent{\bf Principle for the existence proof of Theorem~\ref{th1}~-~}
%-----------------------------------------------------------------------------
The maxwellian~$M$ satisfying the assumptions~\eqref{H1}, \eqref{H2} and~\eqref{H3}, we use the different lemmas proved in Part~\ref{sec-tools}. For instance, using the equivalence between the norms~$\|\cdot\|_{H^1_M}$ and~$\|\cdot\|_{H^1_{M,0}}$ on the space~$H^1_{M,0}$, see Lemma~\ref{lem-poincare}, the operator $-\div \big( M\nabla\big(\frac{\cdot}{M}\big)\big)$ is coerciv on~$H^1_{M,0}$ thus we can (see for instance the Lax-Milgram theorem) prove that there exists a weak solution (that is belonging to~$H^1_{M,0}$) to equations like
\begin{equation*}
-\div\Big( M\nabla\Big( \frac{\psi}{M}\Big)\Big) = f
\end{equation*}
as soon as the source term~$f$ belongs in~$H^{-1}_M$. Moreover in this case we have $\|\psi\|_{H^1_M}\lesssim \|f\|_{H^{-1}_M}$.\par

%----------------------------------------
\noindent Because of the non-coercivity of the operator $-\div \big( M\nabla\big(\frac{\cdot}{M}\big)\big) + \div ( \cdot\, \bkappa )$, we start by studying an approach problem. For each $n\in \N$, let us consider the application $T_n:r\in \R\mapsto \max(\min(r,n),-n)\in \R$ and let us denote by~$F_n$ the following application: $F_n : \tpsi\in L^2_M \mapsto \psi \in H^1_{M,0}\subset L^2_M$ where~$\psi$ is the weak solution of
\begin{equation}\label{pb1-faible-n1}
-\div \Big( M\nabla\Big( \frac{\psi}{M}\Big)\Big) = f - \div\Big( M T_n\Big( \frac{\tpsi}{M} \Big)\bkappa \Big).
\end{equation}
For~$\tpsi\in L^2_M$ we have $MT_n(\nfrac{\tpsi}{M})\in L^2_M$ and since
\footnote
{
Here the assumption on~$\bkappa$ is essential. Following the proof of J.~Droniou~\cite{Droniou} it is possible to improve this assumption using the Sobolev injection~\ref{lem-injection} more finely. See discussion concerning the assumptions on~$\bkappa$ page~\pageref{rema-regul_sur_kappa}.
} 
$\bkappa\in L^\infty(\Omega)$ we get $\div\big( M T_n\big( \frac{\tpsi}{M} \big) \bkappa \big)\in H^{-1}_M$. The function~$F_n$ is then well defined.\par

%----------------------------------------
\noindent Let us prove that~$F_n$ is a compact application by showing that its image~$F_n(L^2_M)$ is bounded in~$H^1_M$. Consider $\psi=F_n(\tpsi)\in F_n(L^2_M)$. Taking~$\psi$ as a test function in the weak formulation of the equation~\eqref{pb1-faible-n1} we obtain
\begin{equation*}
\int_\Omega M \Big|\nabla\Big( \frac{\psi}{M}\Big)\Big|^2 = \langle f,\psi \rangle + \int_\Omega M T_n\Big( \frac{\tpsi}{M} \Big) \bkappa \cdot \nabla\Big( \frac{\psi}{M}\Big).
\end{equation*}
In other words, by using the duality definition and the Cauchy-Schwarz inequality, we have
\begin{equation*}
\|\psi\|_{H^1_{M,0}}^2 \lesssim \|\psi\|_{H^1_{M,0}} + \|\bkappa\|_{L^\infty(\Omega)} \sqrt{\int_\Omega M\Big| T_n\Big( \frac{\tpsi}{M} \Big) \Big|^2} \sqrt{\int_\Omega M \Big|\nabla\Big( \frac{\psi}{M}\Big)\Big|^2}.
\end{equation*}
Using the fact successively that for all~$r\in \R$ we have $|T_n(r)|\leq n$ and that $\int_\Omega M=1$ we deduce that
\begin{equation*}
\|\psi\|_{H^1_{M,0}}^2 
\lesssim
\|\psi\|_{H^1_{M,0}} + n \|\bkappa\|_{L^\infty(\Omega)}\|\psi\|_{H^1_{M,0}}.
\end{equation*}
Consequently we have
\begin{equation*}
\|F_n(\tpsi)\|_{H^1_M} = \|\psi\|_{H^1_{M,0}} \lesssim 1 + n \|\bkappa\|_{L^\infty(\Omega)}.
\end{equation*}
Thus, the image of~$L^2_M$ by the application~$F_n$ is contained in the ball of~$H^1_M$ of radius $1+ n \|\bkappa\|_{L^\infty(\Omega)}$ (up to a multiplicative constant depending on~$\Omega$, which appears in the symbol~$\lesssim$). Moreover, the injection $H^1_M\hookrightarrow L^2_M$ is compact (see Lemma~\ref{lem-compacity}) and the application~$F_n$ is clearly continuous. Applying the Schauder fixed point theorem, we conclude that the application~$F_n$ admits a fixed point, denoted by~$\psi_n$, in~$L^2_M$. This fixed point is consequently a solution of
\begin{equation}\label{pb1-faible-n}
\int_\Omega M \nabla\Big( \frac{\psi_n}{M}\Big) \cdot \nabla\Big( \frac{\fy}{M}\Big) - \int_\Omega M T_n\Big(\frac{\psi_n}{M} \Big) \bkappa \cdot \nabla\Big( \frac{\fy}{M}\Big) = \langle f,\fy \rangle
\end{equation}
for all test functions $\fy \in H^1_{M,0}$.\\
The continuation of the proof consists of obtaining estimates on these functions~$\psi_n$ in order to be able to pass to the limit when~$n$ tends to~$+\infty$.\\

\mathversion{bold}\noindent{\bf Estimate of $M\ln ( 1+|\nfrac{\psi_n}{M}| )$ in $H^1_{M,0}$-norm~-~}\mathversion{normal}
%-----------------------------------------------------------------------------
Let~$\xi$ be the application from~$\R$ to~$\R$ defined by $\xi(r)=\int_0^r \frac{ds}{(1+|s|)^2}$. This application is continuous, piecewise-$\mathcal C^1$  and with a bounded derivative. According to Lemma~\ref{lem-test} we can choose $\fy=M\xi(\nfrac{\psi_n}{M}) - M\int_\Omega M\xi(\nfrac{\psi_n}{M})$ as a test function in formulation~\eqref{pb1-faible-n}.\\
\noindent $\bullet$ The first of the three terms obtained is treated in the following way
\begin{equation}\label{equ2102a}
\begin{aligned}
\int_\Omega M \nabla\Big( \frac{\psi_n}{M}\Big) \cdot \nabla \Big( \xi\Big(\frac{\psi_n}{M}\Big) \Big) 
& =
\int_\Omega M \frac{\Big| \nabla\Big( \frac{\psi_n}{M}\Big) \Big|^2}{\Big( 1+\Big| \frac{\psi_n}{M} \Big|\Big)^2} \\
& =
\Big\| M \ln \Big( 1+\Big| \frac{\psi_n}{M} \Big| \Big) \Big\|_{H^1_{M,0}}^2.
\end{aligned}
\end{equation}
\noindent $\bullet$ For the second term we obtain
\begin{equation*}
\begin{aligned}
\Big| \int_\Omega M T_n\Big(\frac{\psi_n}{M}\Big) \bkappa \cdot \nabla\Big( \xi\Big(\frac{\psi_n}{M}\Big) \Big) \Big|
& =
\Big| \int_\Omega \frac{M T_n\Big(\frac{\psi_n}{M}\Big)}{1+\big| \frac{\psi_n}{M}\big|} \bkappa \cdot \frac{\nabla\Big( \frac{\psi_n}{M}\Big)}{1+\big| \frac{\psi_n}{M}\big|} \Big| \\
& \hspace{-2cm} \leq 
\|\bkappa\|_{L^\infty(\Omega)} \int_\Omega \Big| \frac{T_n \big( \frac{\psi_n}{M} \big)}{1+\big| \frac{\psi_n}{M}\big|} \Big| M \Big| \nabla ( \ln \Big( 1+\Big| \frac{\psi_n}{M}\Big| \Big) ) \Big|.
\end{aligned}
\end{equation*}
Using the fact that for all~$r\in \R$, we have $|T_n(r)|\leq |r|$, we deduce that\footnote{We also use the Cauchy-Schwarz inequality to show that 
$$\int_\Omega Mf \leq \sqrt{\int_\Omega M}\sqrt{\int_\Omega Mf^2} = \sqrt{\int_\Omega Mf^2}.$$}
\begin{equation}\label{equ2102b}
\Big| \int_\Omega M T_n\Big(\frac{\psi_n}{M}\Big) \bkappa \cdot \nabla\Big( \xi\Big(\frac{\psi_n}{M}\Big) \Big) \Big| \lesssim \Big\| M \ln \Big( 1+\Big| \frac{\psi_n}{M} \Big| \Big) \Big\|_{H^1_{M,0}}.
\end{equation}
\noindent $\bullet$ For the last term, using $f\in H^{-1}_M$, we deduce
\begin{equation}\label{equ2102c}
\begin{aligned}
\big| \langle f,\fy \rangle \big|
& \lesssim
\|\fy\|_{H^1_{M,0}} 
=
\sqrt{\int_\Omega M\xi'\Big( \frac{\psi_n}{M} \Big) \nabla\Big( \frac{\psi_n}{M} \Big) }\\
& =
\sqrt{ \int_\Omega M \frac{\big| \nabla\big( \frac{\psi_n}{M}\big) \big|^2}{\big( 1+\big| \frac{\psi_n}{M} \big|\big)^2} }
=
\Big\| M \ln \Big( 1+\Big| \frac{\psi_n}{M} \Big| \Big) \Big\|_{H^1_{M,0}}.
\end{aligned}
\end{equation}
The three estimates~\eqref{equ2102a},~\eqref{equ2102b} and~\eqref{equ2102c} enable us to obtain for all~$n\in \N$
\begin{equation}\label{estimate1}
\Big\| M \ln \Big( 1+\Big| \frac{\psi_n}{M} \Big| \Big) \Big\|_{H^1_{M,0}} \lesssim 1.
\end{equation}

\mathversion{bold}\noindent{\bf Estimate of $\dsp \mu(\{\Q\in \Omega\ ;\ |\psi_n(\Q)|\geq kM(\Q)\})$~-~}\mathversion{normal}
%-----------------------------------------------------------------------------
In this paragraph, we control the size of the set where~$\psi_n$ has a large value, that is the set $\dsp \mathcal E_k = \{\Q\in \Omega\ ;\ |\psi_n(\Q)|\geq kM(\Q)\}$ for~$k\in \N$.
The natural measure in the present context is the measure $\mathrm d\mu = M(\Q) \mathrm d\Q$ ($d\Q$ being the classical Lebesgue measure on $\Omega\subset \R^d$), which enables to take into account the weight of the Maxwellian~$M$.\par
\noindent Writing $\mathcal E_k = \{\Q\in \Omega\ ;\ (\ln(1+|\nfrac{\psi_n(\Q)}{M(\Q)}|))^2 \geq (\ln(1+k))^2\}$ we obtain
\begin{equation*}
\int_\Omega M \Big(\ln(1+\Big|\frac{\psi_n}{M}\Big|)\Big)^2 = \int_{\mathcal E_k} M \Big(\ln(1+\Big|\frac{\psi_n}{M}\Big|)\Big)^2 + \int_{\Omega\setminus \mathcal E_k} M \Big(\ln(1+\Big|\frac{\psi_n}{M}\Big|)\Big)^2.
\end{equation*}
We easily deduce the following estimate
\begin{equation*}
\int_\Omega M \Big(\ln(1+\Big|\frac{\psi_n}{M}\Big|)\Big)^2 \geq \int_{\mathcal E_k} M \Big(\ln(1+\Big|\frac{\psi_n}{M}\Big|)\Big)^2  \geq \int_{\mathcal E_k} M (\ln(1+k))^2.
\end{equation*}
Introducing the measure $\mathrm d\mu = M(\Q) \mathrm d\Q$ this inequality is also rewritten
\begin{equation*}
\mu(\mathcal E_k) \leq \frac{1}{(\ln(1+k))^2} \Big\| M \ln \Big( 1+\Big| \frac{\psi_n}{M} \Big| \Big) \Big\|_{L^2_M}^2.
\end{equation*}
Taking into account the estimate~\eqref{estimate1}, the previous equation is written
\begin{equation}\label{estimate2}
\mu(\{\Q\in \Omega\ ;\ |\psi_n(\Q)|\geq kM(\Q)\}) \lesssim \frac{1}{(\ln(1+k))^2}.
\end{equation}

%-----------------------------------------------------------------------------
\mathversion{bold}\noindent{\bf Estimate of $M S_k(\nfrac{\psi_n}{M})$ in $H^1_{M,0}$-norm~-~}\mathversion{normal}
%-----------------------------------------------------------------------------
\noindent Recall that for $k\in \N$ the application~$T_k$ is given by $T_k:r\in \R\mapsto \max(\min(r,k),-k)\in \R$. We now define the application~$S_k$ such that $T_k + S_k = \mathrm{id}$.
To obtain an estimate on~$\psi_n$ we successively obtain an estimate on $M S_k(\nfrac{\psi_n}{M})$ and then on $M T_k(\nfrac{\psi_n}{M})$ for a sufficiently large~$k\in \N$.\par
%
%-----------------------------------------------------------------------------
\noindent Let $k\in \N$. Taking $\fy=MS_k(\nfrac{\psi_n}{M}) - M \int_\Omega MS_k(\nfrac{\psi_n}{M})$ as a test function test in~\eqref{pb1-faible-n}. According to Lemma~\ref{lem-test}, this choice is possible and we obtain
\begin{equation*}\label{pb1-faible-n-Sk}
\underbrace{
\int_\Omega M \nabla\Big( \frac{\psi_n}{M}\Big) \cdot \nabla \Big( S_k\Big( \frac{\psi_n}{M}\Big) \Big)
}_{A}
- \underbrace{
\int_\Omega  MT_n\Big(\frac{\psi_n}{M}\Big) \bkappa \cdot \nabla \Big( S_k\Big( \frac{\psi_n}{M}\Big) \Big)
}_{B}
=
\underbrace{\phantom{\Big|}\langle f , \fy \rangle\phantom{\Big|}}_{C}.
\end{equation*}
\noindent $\bullet$ Since $S_k+T_k=\mathrm{id}$ and for all~$r\in \R$ we have $S_k'(r)=0$ or $T_k'(r)=0$ we deduce that the first term~$A$ is written
\begin{equation}\label{equ-a2250}
A
=
\int_\Omega M \Big| \nabla \Big( S_k\Big( \frac{\psi_n}{M}\Big) \Big) \Big|^2
=
\Big\|M S_k\Big( \frac{\psi_n}{M}\Big) \Big\|_{H^1_{M,0}}^2.
\end{equation}
\noindent $\bullet$ Using the fact that for all $r\in \R$ we have $|T_n(r)|\leq |r|$ and using the Cauchy-Schwarz inequality, we estimate the second term~$B$ in the following way
\begin{equation*}
|B|
\leq
\|\bkappa\|_{L^\infty(\Omega)} \sqrt{\int_\Omega \frac{|\psi_n|^2}{M}}\sqrt{\int_\Omega M\Big| \nabla \Big( S_k\Big( \frac{\psi_n}{M}\Big) \Big) \Big|^2}.
\end{equation*}
However $| \nfrac{\psi_n}{M} | = | T_k(\nfrac{\psi_n}{M}) + S_k(\nfrac{\psi_n}{M}) | \leq k + |S_k(\nfrac{\psi_n}{M}) |$ thus $| \nfrac{\psi_n}{\sqrt{M}} | \leq k\sqrt{M} + \sqrt{M}|S_k(\nfrac{\psi_n}{M}) |$ and using the triangular inequality we obtain
\begin{equation*}
\sqrt{\int_\Omega \frac{|\psi_n|^2}{M}} \leq \sqrt{\int_\Omega k^2 M} + \sqrt{\int_\Omega M\Big| S_k\Big(\frac{\psi_n}{M}\Big) \Big|^2} = k + \Big\|M S_k\Big( \frac{\psi_n}{M}\Big) \Big\|_{L^2_M}.
\end{equation*}
Since $S_k(r)=0$ for $|r|<k$, we can estimate this last term as follows:
\begin{equation*}
\Big\|M S_k\Big( \frac{\psi_n}{M}\Big) \Big\|_{L^2_M}^2 = \int_\Omega M\Big| S_k\Big(\frac{\psi_n}{M}\Big) \Big|^2 = \int_{\mathcal E_k} M\Big| S_k\Big(\frac{\psi_n}{M}\Big) \Big|^2,
\end{equation*}
where we recall that $\mathcal E_k=\{\Q\in \Omega\ ;\ |\psi_n(\Q)|\geq kM(\Q)\}$. According to the Hölder inequality, for all~$p>1$, denoting by~$q$ the conjugate of~$p$ (i.e. such that $\frac{1}{p}+\frac{1}{q}=1$) and using the estimate~\eqref{estimate2}, we obtain
\begin{equation*}
\begin{aligned}
\Big\|M S_k\Big( \frac{\psi_n}{M}\Big) \Big\|_{L^2_M}^2 
& \leq
\Big( \int_{\mathcal E_k} M \Big)^{1/q} \Big( \int_{\mathcal E_k} M\Big| S_k\Big(\frac{\psi_n}{M}\Big) \Big|^{2p} \Big)^{1/p} \\
& \lesssim
\frac{1}{(\ln(1+k))^{2/q}} \Big( \int_\Omega M\Big| S_k\Big(\frac{\psi_n}{M}\Big) \Big|^{2p} \Big)^{1/p}.
\end{aligned}
\end{equation*}
We thus control the~$L^2_M$-norm of $M S_k(\nfrac{\psi_n}{M})$ using his~$L^{2p}_M$-norm . But this~$L^{2p}_M$-norms can itself be controlled, for an adapted value of~$p$ by the~$H^1_M$-norm.
In fact, using the weighted Sobolev embedding (see Lemma~\ref{lem-injection}) there exists $p>1$ for which we have the inequality
\begin{equation*}
\begin{aligned}
\Big( \int_\Omega M\Big| S_k\Big(\frac{\psi_n}{M}\Big) \Big|^{2p} \Big)^{1/p}
& \lesssim
\Big\|M S_k\Big( \frac{\psi_n}{M}\Big) \Big\|_{H^1_M}^2 \\
& \lesssim
\Big\|M S_k\Big( \frac{\psi_n}{M}\Big) \Big\|_{H^1_{M,0}}^2 + \Big\|M S_k\Big( \frac{\psi_n}{M}\Big) \Big\|_{L^2_M}^2.
\end{aligned}
\end{equation*}
We deduce a control on the~$L^2_M$-norm of $M S_k(\nfrac{\psi_n}{M})$ using his~$H^1_{M,0}$-norm:
\begin{equation*}
\Big( 1-\frac{1}{(\ln(1+k))^{2/q}} \Big) \Big\|M S_k\Big( \frac{\psi_n}{M}\Big) \Big\|_{L^2_M}^2 
\lesssim
\frac{1}{(\ln(1+k))^{2/q}} \Big\|M S_k\Big( \frac{\psi_n}{M}\Big) \Big\|_{H^1_{M,0}}^2,
\end{equation*}
that is a control of the form $\big\|M S_k\big( \nfrac{\psi_n}{M}\big) \big\|_{L^2_M} \leq \mathcal A(k) \big\|M S_k\big( \nfrac{\psi_n}{M}\big) \big\|_{H^1_{M,0}}$ where $\mathcal A(k)$ tends to~$0$ when~$k$ tends to~$+\infty$. Hence, we obtain the following estimate for the term~$B$ of the left hand side of equation~\eqref{pb1-faible-n-Sk}:
\begin{equation}\label{equ-b2250}
|B|
\leq
\|\bkappa\|_{L^\infty(\Omega)} \Big( k + \mathcal A(k) \Big\|M S_k\Big( \frac{\psi_n}{M}\Big) \Big\|_{H^1_{M,0}} \Big) \Big\|M S_k\Big( \frac{\psi_n}{M}\Big) \Big\|_{H^1_{M,0}}.
\end{equation}
\noindent $\bullet$ The last term of the equation~\eqref{pb1-faible-n-Sk} is controlled as follow
\begin{equation}\label{equ-c2250}
\begin{aligned}
|C| = | \langle f,\fy \rangle |
& \lesssim
\|\fy\|_{H^1_{M,0}}
=
\sqrt{\int_\Omega M\Big| \nabla\Big(\frac{\fy}{M}\Big) \Big|^2} \\
& \lesssim
\sqrt{\int_\Omega M\Big| \nabla\Big(S_k\Big(\frac{\fy}{M}\Big)\Big) \Big|^2} 
=
\Big\| M S_k\Big(\frac{\psi_n}{M}\Big) \Big\|_{H^1_{M,0}}.
\end{aligned}
\end{equation}
The previous estimates~\eqref{equ-a2250}, \eqref{equ-b2250} and~\eqref{equ-c2250} enable the deduction, from equation~\eqref{pb1-faible-n-Sk}, for all $k\in \N$, of the following inequality
\begin{equation*}
\Big\|M S_k\Big( \frac{\psi_n}{M}\Big) \Big\|_{H^1_{M,0}} 
\lesssim
\|\bkappa\|_{L^\infty(\Omega)} \Big( k + \mathcal A(k) \Big\|M S_k\Big( \frac{\psi_n}{M}\Big) \Big\|_{H^1_{M,0}} \Big) + 1.
\end{equation*}
Since $\mathcal A(k)$ tends to~$0$ when~$k$ tends to~$+\infty$, it possible to obtain for a sufficiently large~$k$, the inequality
\begin{equation}\label{estimate3}
\Big\| M S_k\Big(\frac{\psi_n}{M}\Big) \Big\|_{H^1_{M,0}} \lesssim 1.
\end{equation}

\mathversion{bold}\noindent{\bf Estimate of $M T_k(\nfrac{\psi_n}{M})$ in $H^1_{M,0}$-norm~-~}\mathversion{normal}
%-----------------------------------------------------------------------------
Choose now $\fy=M T_k(\nfrac{\psi_n}{M}) - M\int_\Omega M T_k(\nfrac{\psi_n}{M})$ as a test function in equation~\eqref{pb1-faible-n} (according to Lemma~\ref{lem-test} we have $\fy\in H^1_{M,0}$). As for the estimate of $M S_k(\nfrac{\psi_n}{M})$, we study each of three terms, named $A$, $B$ and $C$ as previously, present in equation~\eqref{pb1-faible-n}.\\
$\bullet$ The first is written
\begin{equation*}%\label{equ2114a}
A%\int_\Omega M \nabla\Big( \frac{\psi_n}{M}\Big) \cdot \nabla \Big( T_k\Big( \frac{\psi_n}{M}\Big) \Big)
=
\int_\Omega M \Big| \nabla \Big( T_k\Big( \frac{\psi_n}{M}\Big) \Big) \Big|^2
=
\Big\|M T_k\Big( \frac{\psi_n}{M}\Big) \Big\|_{H^1_{M,0}}^2.
\end{equation*}
$\bullet$ For the second term, we proceed as follow:
\begin{equation*}
\begin{aligned}
|B|%\Big| \int_\Omega M T_n\Big(\frac{\psi_n}{M} \Big) \bkappa \cdot \nabla \Big( T_k\Big( \frac{\psi_n}{M}\Big) \Big) \Big| 
& \leq
\|\bkappa\|_{L^\infty(\Omega)} \int_\Omega M \Big| T_n\Big(\frac{\psi_n}{M} \Big)\Big| |\nabla \Big( T_k\Big( \frac{\psi_n}{M}\Big) \Big)| \\
& \leq
\|\bkappa\|_{L^\infty(\Omega)} \int_\Omega |\psi_n| |\nabla \Big( T_k\Big( \frac{\psi_n}{M}\Big) \Big) |.
\end{aligned}
\end{equation*}
But for $|\nfrac{\psi_n}{M}|\geq k$ we have $\nabla ( T_k( \nfrac{\psi_n}{M})) = 0$ whereas for $|\nfrac{\psi_n}{M}| < k$ we clearly have $|\psi_n|<kM$ and consequently, according to the Cauchy-Schwarz inequality we obtain
\begin{equation*}
\begin{aligned}
|B|%\Big| \int_\Omega M T_n\Big(\frac{\psi_n}{M} \Big) \bkappa \cdot \nabla \Big( T_k\Big( \frac{\psi_n}{M}\Big) \Big) \Big| 
& \leq
\|\bkappa\|_{L^\infty(\Omega)} \int_\Omega k M \Big|\nabla \Big( T_k\Big( \frac{\psi_n}{M}\Big) \Big)\Big| \\
& \leq 
\|\bkappa\|_{L^\infty(\Omega)} \sqrt{\int_\Omega k^2 M} \sqrt{\int_\Omega M \Big|\nabla \Big( T_k\Big( \frac{\psi_n}{M}\Big) \Big)\Big|^2} \\
& \leq
k \|\bkappa\|_{L^\infty(\Omega)} \Big\|M T_k\Big( \frac{\psi_n}{M}\Big) \Big\|_{H^1_{M,0}}.
\end{aligned}
\end{equation*}
\noindent $\bullet$ The last term is treated like those of the previous estimates:
\begin{equation*}%\label{equ2114c}
|C| = | \langle f , \fy \rangle |
\lesssim
\Big\|M T_k\Big( \frac{\psi_n}{M}\Big) \Big\|_{H^1_{M,0}}.
\end{equation*}
These three estimates give (note that this estimate depends on~$k$, but that~$k$ has been fixed)%~\eqref{equ2114a},~\eqref{equ2114b} and~\eqref{equ2114c} give
\begin{equation}\label{estimate4}
\Big\| M T_k\Big(\frac{\psi_n}{M}\Big) \Big\|_{H^1_{M,0}} \lesssim 1 .
\end{equation}

\mathversion{bold}\noindent{\bf Estimate of $\psi_n$ in $H^1_{M,0}$~-~}\mathversion{normal}
%-----------------------------------------------------------------------------
Since for all $k\in \N$ we have $S_k+T_k=\mathrm{id}$ we obtain
\begin{equation*}
\begin{aligned}
\|\psi_n\|_{H^1_{M,0}}
& =
\Big\|MS_k\Big(\frac{\psi_n}{M}\Big) + MT_k\Big(\frac{\psi_n}{M}\Big)\Big\|_{H^1_{M,0}} \\
& \leq
\Big\|MS_k\Big(\frac{\psi_n}{M}\Big)\Big\|_{H^1_{M,0}} + \Big\|MT_k\Big(\frac{\psi_n}{M}\Big)\Big\|_{H^1_{M,0}}.
\end{aligned}
\end{equation*}
Using the estimates~\eqref{estimate3} and~\eqref{estimate4} we deduce that for all $n\in \N$ we have
\begin{equation}\label{estimate5}
\|\psi_n\|_{H^1_{M,0}} \lesssim 1.
\end{equation}

\mathversion{bold}\noindent{\bf Convergence of the sequence $\{\psi_n\}_{n\in \N}$~-~}\mathversion{normal}
%-----------------------------------------------------------------------------
According to the estimate~\eqref{estimate5} the sequence $\{\psi_n\}_{n\in \N}$ is bounded in~$H^1_{M,0}$. According to the Lemma~\ref{lem-compacity} a subsequence of the sequence $\{\psi_n\}_{n\in \N}$ (always denoted by $\{\psi_n\}_{n\in \N}$) admits a limit~$\psi$, weak in~$H^1_{M,0}$ and strong in~$L^2_M$. In order to perform the limit in equation~\eqref{pb1-faible-n}, it is sufficient to prove that the sequence $\{MT_n(\nfrac{\psi_n}{M})\}_{n\in \N}$ tends to~$\psi$ in~$L^2_M$. We obtain
\begin{equation*}
\Big\|MT_n\Big(\frac{\psi_n}{M}\Big)-\psi\Big\|_{L^2_M}^2 
\leq
\Big\|MT_n\Big(\frac{\psi_n}{M}\Big)-MT_n\Big(\frac{\psi}{M}\Big)\Big\|_{L^2_M}^2 
+
\Big\|MT_n\Big(\frac{\psi}{M}\Big)-\psi\Big\|_{L^2_M}^2.
\end{equation*}
However the application $T:\R\rightarrow \R$ is $1$-lipschitz and we have
\begin{equation*}
\begin{aligned}
\Big\|MT_n\Big(\frac{\psi_n}{M}\Big)-MT_n\Big(\frac{\psi}{M}\Big)\Big\|_{L^2_M}^2 
& =
\int_\Omega M \Big| T_n\Big(\frac{\psi_n}{M}\Big) - T_n\Big(\frac{\psi}{M}\Big) \Big|^2 \\
& \leq 
\int_\Omega M \Big| \frac{\psi_n}{M} - \frac{\psi}{M} \Big|^2
=
\|\psi_n-\psi\|_{L^2_M}^2,
\end{aligned}
\end{equation*}
which proves that $\|MT_n(\nfrac{\psi_n}{M})-MT_n(\nfrac{\psi}{M})\|_{L^2_M}$ tends to~$0$ when~$n$ tends to~$+\infty$. As regards the other term, the Lebesgue convergence dominated theorem directly affirms that $\|MT_n(\nfrac{\psi}{M})-\psi\|_{L^2_M}$ also tends to~$0$ when~$n$ tends to~$+\infty$. Finally, it was shown that the sequence $\{MT_n(\nfrac{\psi_n}{M})\}_{n\in \N}$ converges to~$\psi$ in~$L^2_M$ and consequently that~$\psi$ is a solution of
\begin{equation*}
\int_\Omega M \nabla\Big( \frac{\psi}{M}\Big) \cdot \nabla\Big( \frac{\fy}{M}\Big) - \int_\Omega \psi \bkappa \cdot \nabla\Big( \frac{\fy}{M}\Big) = \langle f,\fy \rangle \qquad \forall \fy\in H^1_{M,0}.
\end{equation*}

\subsection{Uniqueness proof in Theorem~\ref{th1}}
%%%%%%%%%%%%%%%%%%%%%%%%%%%%%%%%%%%%%%%%%%%%%%%%%%%%%%%%%%%%%%%%%%%%%%%%%%%%%%

\noindent{\bf Main steps for the uniqueness proof~-~}
%-----------------------------------------------------------------------------
To prove uniqueness, we proceed as follows: We start by introducing the dual problem. It is shown that this dual problem admits a solution by using the Schauder topological degree method. Then, by using the existence both problem and its dual, we deduce uniqueness from these two problems.\\

\noindent{\bf Introduction of the dual problem~-~}
%-----------------------------------------------------------------------------
For $g\in H^{-1}_M$ let us consider the elliptic partial differential equation
\begin{equation}\label{pb1-1dual}
-\div\Big( M\nabla\Big( \frac{\phi}{M}\Big)\Big) -  M\bkappa\cdot \nabla \Big( \frac{\phi}{M} \Big) = g 
\quad \text{on~$\Omega$}
\end{equation}
and we look for a solution~$\phi\in H^1_{M,0}$ to this equation.\\

\noindent{\bf A compact application for the dual problem~-~}
%-----------------------------------------------------------------------------
For $\tphi \in H^1_{M,0}$ we have $M\bkappa\cdot \nabla ( \nfrac{\tphi}{M} )\in L^2_M \subset H^{-1}_M$ since $\| M\bkappa\cdot \nabla ( \nfrac{\tphi}{M} ) \|_{L^2_M} \leq \|\bkappa\|_{L^\infty(\Omega)} \|\tphi\|_{H^1_{M,0}}$. Since the operator $\fy\mapsto -\div ( M\nabla(\nfrac{\fy}{M}))$ is coerciv in~$H^1_{M,0}$, there exists thus a unique solution $\phi = G(\tphi) \in H^1_{M,0}$ such that for all $\fy\in H^1_{M,0}$
\begin{equation}\label{pb1-faibledual1}
\int_\Omega M \nabla\Big( \frac{\phi}{M}\Big) \cdot \nabla\Big( \frac{\fy}{M}\Big) - \int_\Omega \fy\bkappa\cdot \nabla \Big( \frac{\tphi}{M} \Big) 
=
\langle g,\fy \rangle.
\end{equation}
%-------------------------------
\noindent This defines an application $G:H^1_{M,0} \rightarrow H^1_{M,0}$. It is quite easy to see that~$G$ is continuous; indeed, if~$\tphi_n$ tends to~$\tphi$ in~$H^1_{M,0}$ then $M\bkappa\cdot \nabla ( \nfrac{\tphi_n}{M} )$ tends to $M\bkappa\cdot \nabla ( \nfrac{\tphi}{M} )$ in~$H^{-1}_M$ (more precisely in~$L^2_M$). Thus $\div ( M\nabla ( \nfrac{\phi_n}{M} ) )$ tends to $\div ( M\nabla ( \nfrac{\phi}{M} ) )$ which implies that $\phi_n=G(\tphi_n)$ tends to $\phi=G(\tphi)$ in~$H^1_{M,0}$.\par

%-------------------------------
\noindent We will now prove that~$G$ is a compact operator. Suppose that the sequence $\{\tphi_n\}_{n\in \N}$ is bounded in~$H^1_{M,0}$; then $\{M\bkappa\cdot \nabla ( \nfrac{\tphi_n}{M} )\}_{n\in \N}$ is bounded in~$H^{-1}_M$ so that, using $\fy=G(\tphi_n)=\phi_n$ as a test function in the equation satisfied by~$\phi_n$, we get using the Lemma~\ref{lem-poincare}
\begin{equation*}
\|\phi_n\|_{H^1_{M,0}}^2 \lesssim \Big(1+\Big\|M\bkappa\cdot \nabla \Big( \frac{\tphi_n}{M} \Big)\Big\|_{H^{-1}_M}\Big)\|\tphi_n\|_{H^1_{M,0}},
\end{equation*}
which implies that the sequence~$\{\phi_n\}_{n\in \N}$ is bounded in~$H^1_{M,0}$. Using the Lemma~\ref{lem-compacity}, up to a subsequence, we can thus suppose that~$\{\phi_n\}_{n\in \N}$ converges a.e. on~$\Omega$ and is bounded in~$L^2_M$. Let $(n,m)\in \N^2$; subtract the equation satisfied by~$\phi_m$ to the equation satisfied by~$\phi_n$ and use $\fy=\phi_n-\phi_m$ as a test function, this gives using the Lemma~\ref{lem-poincare} again
\begin{equation*}
\|\phi_n-\phi_m\|_{H^1_{M,0}}^2 
\leq
\Big| \int_\Omega (\phi_n-\phi_m) \bkappa\cdot \nabla \Big( \frac{\tphi_n-\tphi_m}{M}\Big) \Big| 
\lesssim
\|\phi_n-\phi_m\|_{L^2_M}.
\end{equation*}
From the strong convergence of~$\{\phi_n\}_{n\in \N}$ to~$\phi$ in~$L^2_M$ we deduce that the sequence~$\{\phi_n\}_{n\in \N}$ is a Cauchy sequence in~$H^1_{M,0}$ and converges in this space. We deduce that the application~$G$ is compact.\\

\noindent{\bf Existence result for the dual problem using the Leray-Schauder topological degree~-~}
%-----------------------------------------------------------------------------
\noindent According to the Leray-Schauder topological theory (see the founder article of J. Leray and J. Schauder~\cite{Leray}) since the operator~$G$ introduced with equation~\eqref{pb1-faibledual1} is a compact operator, to prove that it has a fixed point, we just have to find $R>0$ such that for all $s\in [0,1]$ there exists no solution of $\phi-sG(\phi) = 0$ satisfying $\|\phi\|_{H^1_{M,0}}=R$.\\
Let $s\in [0,1]$ and suppose that $\phi \in H^1_{M,0}$ satisfies $\phi=sG(\phi)$. We have for all $\fy\in H^1_{M,0}$
\begin{equation}\label{pb1-faibledual12}
\int_\Omega M \nabla\Big( \frac{\phi}{M}\Big) \cdot \nabla\Big( \frac{\fy}{M}\Big) - s\int_\Omega \fy \bkappa\cdot \nabla \Big( \frac{\phi}{M} \Big) = \langle sg,\fy \rangle.
\end{equation}
Using the ``non-dual'' problem (see the existence proof of Theorem~\ref{th1} where we obtain an existence solution of equation~\eqref{pb1-faible-n}), we know that for all $f\in H^{-1}_M$ there exists at least one solution $\psi\in H^1_{M,0}$ such that for all $\fy\in H^1_{M,0}$
\begin{equation}\label{pb1-faible12}
\int_\Omega M \nabla\Big( \frac{\psi}{M}\Big) \cdot \nabla\Big( \frac{\fy}{M}\Big) 
-
s\int_\Omega \psi \bkappa \cdot \nabla\Big( \frac{\fy}{M}\Big) 
=
\langle f,\fy \rangle.
\end{equation}
Moreover, according to estimate~\eqref{estimate5} there exists $C_1\in \R^+$ such that for all $f\in H^{-1}_M$ with $\|f\|_{H^{-1}_M}\leq 1$ and for all $s\in [0,1]$ we have $\|\psi\|_{H^1_{M,0}} \leq C_1$. We can verify that this constant~$C_1$ depends only on $\|f\|_{H^{-1}_M}$ and can be selected independently on the function~$f$ when $\|f\|_{H^{-1}_M}\leq 1$. In addition according to the estimates obtained in the existence proof of Theorem~\ref{th1} this constant~$C_1$ depends on $\|s \bkappa\|_{L^\infty(\Omega)}\leq \|\bkappa\|_{L^\infty(\Omega)}$ and consequently the constant~$C_1$ can also be selected independently of~$s$.\par
\noindent By taking $\fy=\phi$ in the equation~\eqref{pb1-faible12} satisfied by~$\psi$ and $\fy=\psi$ in the equation~\eqref{pb1-faibledual12} satisfied by~$\phi$, we obtain
\begin{equation*}
\langle f,\phi \rangle = \langle sg,\psi \rangle \leq s \|g\|_{H^{-1}_M} C_1 \leq \|g\|_{H^{-1}_M} C_1 := C_2.
\end{equation*}
Since this inequality is satisfied for all $f\in H^{-1}_M$ such that $\|f\|_{H^{-1}_M}\leq 1$, we deduce that $\|\phi\|_{H^1_{M,0}} \leq C_2$.\par
\noindent Now take $R=C_2+1$. We have just proven that, for any $s\in[0,1]$, any solution of $\phi-sG(\phi) = 0$ satisfies $\|\phi\|_{H^1_{M,0}} < R$; thus by the Leray-Schauder topological degree theory, the application~$G$ has a fixed point, that is to say a solution of~\eqref{pb1-faibledual1}.\\

\noindent{\bf Uniqueness~-~}
%-----------------------------------------------------------------------------
Since the equation~\eqref{pb1-faible} is linear, it is sufficient to prove that the only solution of~\eqref{pb1-faible} without source term, i.e. taking $f=0$, is the null function. Let~$\psi$ be a solution of~\eqref{pb1-faible} with $f=0$ and let~$\phi$ be a solution of~\eqref{pb1-1dual} with
%\footnote{For each $\psi \in H^1_M$, the function $\mathrm{sign}(\psi)$ is defined as follow : for all $\fy\in H^1_M$
%\begin{equation*}
%\langle \mathrm{sign}(\psi),\fy \rangle = \int_{\{\Q\in \Omega \ ;\ \psi(\Q)>0\}} \fy - \int_{\{\Q\in \Omega \ ;\ \psi(\Q)<0\}} \fy.
%\end{equation*}
%We verify that this linear form on~$H^1_M$ is continuous since thanks to the Cauchy-Schwarz inequality we get
%\begin{equation*}
%|\langle \mathrm{sign}(\psi),\fy \rangle|
%\leq 2\int_\Omega |\fy| 
%\leq 2 \sqrt{\int_\Omega \frac{|\fy|^2}{M}} \sqrt{\int_\Omega M} 
%= 2\|\fy\|_{L^2_M} 
%\leq 2\|\fy\|_{H^1_{M,0}}.
%\end{equation*}
%}
$g=\mathrm{sign}(\psi) \in H^{-1}_M$.
By putting $\fy=\phi$ as a test function in the equation~\eqref{pb1-faible} satisfied by~$\psi$ and $\fy=\psi$ as a test function in the weak formulation of the equation~\eqref{pb1-1dual} satisfied by~$\phi$, we respectively get
\begin{equation*}
\begin{aligned}
& \int_\Omega M \nabla\Big( \frac{\psi}{M}\Big) \cdot \nabla\Big( \frac{\phi}{M}\Big) - \int_\Omega \psi \bkappa \cdot \nabla\Big( \frac{\phi}{M}\Big) = 0 \qquad \text{and}\\
& \int_\Omega M \nabla\Big( \frac{\phi}{M}\Big) \cdot \nabla\Big( \frac{\psi}{M}\Big) - \int_\Omega \psi \bkappa \cdot \nabla \Big( \frac{\phi}{M} \Big) = \langle \mathrm{sign}(\psi),\psi \rangle.
\end{aligned}
\end{equation*}
We deduce that $\langle \mathrm{sign}(\psi),\psi \rangle = 0$, that is to say $\int_\Omega |\psi| = 0$ and then $\psi=0$.

\begin{rema}
A similar reasoning gives the uniqueness of the solution of the dual problem~\eqref{pb1-faibledual1}.
\end{rema}

%%%%%%%%%%%%%%%%%%%%%%%%%%%%%%%%%%%%%%%%%%%%%%%%%%%%%%%%%%%%%%%%%%%%%%%%%%%%%%
\section{Application to fluid mechanics}\label{sec-numeric}
%%%%%%%%%%%%%%%%%%%%%%%%%%%%%%%%%%%%%%%%%%%%%%%%%%%%%%%%%%%%%%%%%%%%%%%%%%%%%%

\subsection{The FENE model for dilute polymers}\label{sub3-1}
%%%%%%%%%%%%%%%%%%%%%%%%%%%%%%%%%%%%%%%%%%%%%%%%%%%%%%%%%%%%%%%%%%%%%%%%%%%%%%

\noindent A natural framework where vectors fields strongly explode at the boundaries of a domain is the framework of the modeling of the spring whose extension is finite (that is physically realist). In fluid mechanics, such an approach is used to develop polymer models in solution. It is this point of view which we have chooses to present in order to illustrate the preceding theoretical study.\\
The simplest micro-mechanical approach to model the polymer molecules in a dilute solution is the dumbbell model in which the polymers are represented by two beads connected by a spring. The configuration vector~$\Q$ describes the orientation and the elongation of such a dumbbell~\cite{Lelievre1,Chauviere}. The force of the spring is governed by some law that should be derived from physical arguments. We choose here the popular FENE model, in which the maximum extensibility of the dumbbell is fixed at some value determined by the dimensionless parameter~$\ell$ and the spring force takes the simple form
\begin{equation*}
\E(\Q) = \frac{\Q}{1-\nfrac{|\Q|^2}{\ell^2}}.
\end{equation*}
The configuration vector~$\Q$ depends on time~$t$ and macroscopic position of the dumbbell~$\x$ in the flow. Moreover, it satisfies the following stochastic differential equation (see~\cite{Ottinger} for details):
\begin{equation}\label{EDS}
d\Q = \big( (\nabla \u)^T \cdot\Q - \frac{1}{2\De} \E(\Q) \big)\, dt + \sqrt{\frac{1}{2\De}} \,d\W,
\end{equation}
where the $2$-tensor~$(\nabla \u)^T$ is the transposed velocity gradient, $\De$ is a dimensionless number called the Deborah number (linked to the relaxation time of the fluid) and~$\W$ is the Wiener random process that accounts for the Brownian forces acting on each bead. Equation~\eqref{EDS} should be understood as the Itô ordinary stochastic differential equations along the particle paths since the dumbbells'centers of mass are supposed on average to follow the particules of the solvent fluid.\\
As is well known (see Section 3.3 of~\cite{Ottinger}), every Itô ordinary stochastic differential equation can be associated with a partial differential equation for the probability density function $\fy(t,\x,\Q)$ of the random process $\Q(t,\x)$. In particular, equation~\eqref{EDS} implies the following, also called Fokker-Planck, equation for $\fy(t,\x,\Q)$:
\begin{equation}\label{FP1357}
\frac{\partial \fy}{\partial t} + \u\cdot \nabla_\x \fy 
= 
\frac{1}{2\De} \Delta_\Q \fy - \div_\Q \Big(\fy \Big( (\nabla \u)^T \cdot \Q - \frac{1}{2\De} \E \Big) \Big).
\end{equation}
%
%----------------------------------------
\begin{figure}[htbp]
\begin{center}
{\psfrag{Flow}{Flow}\psfrag{Physical domain}{Physical domain}
\includegraphics[width=10cm]{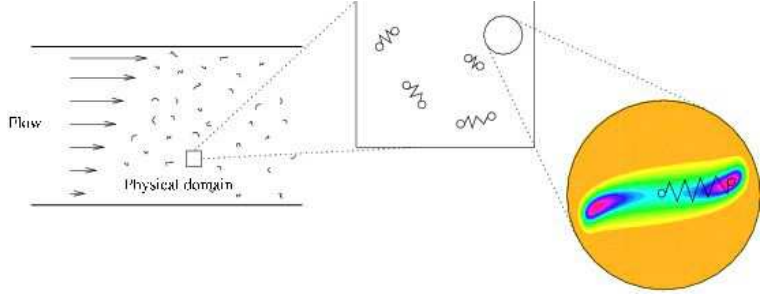}
}
\caption{On the left, we have drawn the physical domain of a real flow for a diluted polymers solution. From the microscopic point of view, the polymer chains are identified to independent mass-springs systems (called dumbbells). The orientation and the length of each dumbbell is governed by a quantity (denoted by~$\fy$ in this paper) distributed in a ball whose radius corresponds to the maximum extension of the spring. On the right, the colors correspond to the various probabilities that dumbbell be in the given position. For instance, the drawn dumbbell is the dumbell which has the most chance to be present (with its ``symmetrical'' compared to the center of the ball).}%\label{}
\end{center}
\end{figure}
%----------------------------------------
%

\noindent In certain modes the dominating terms correspond to the terms of the right-hand side member of equation~\eqref{FP1357}. It is the case, for instance, when the flow is supposed to be thin, see~\cite{Chupin}. In these configurations, the distribution~$\fy$ can be seen like depending only on~$\Q$ (to be rigorous, $\fy$ also depends on time~$t$ and on the macroscopic position~$\x$, {\it via} the presence of the gradient $\nabla \u(t, \x)$ but these dependences can be seen as parameters) and the equation~\eqref{FP1357} is approached by the following Fokker-Planck equation on~$\fy(\Q)$:
\begin{equation}\label{FP1436}
-\Delta \fy + \div(\fy \, \F) = 0 \quad \text{in $B(0,\ell)$},
\end{equation}
with $\F = 2\De\, (\nabla \u)^T \cdot \Q - \E$. It exactly corresponds to those studied in the first part of this paper (see equation~\eqref{FP}) in the case $\Omega=B(0,\ell)$ and without source term: $f=0$.\\
Although it is wished that the solution $\fy(\Q)$ cancels for values~$\Q$ such that $|\Q|=\ell$ (i.e. we wishe that the maximum length of the springs is $\ell$ and that there is no spring of this length), the classical framework of the Theorem~\ref{th-Droniou} does not correspond to this equation provided with the homogeneous Dirichlet boundary conditions. In fact, the force~$\F$ is not sufficiently regular: we have $\F\notin L^d(B(0,\ell))$. Roughly speaking, the FENE model takes into account the finite extensibility of the polymer chain, through an important explosive force when~$|\Q|$ tends to~$\ell$.\\
On the other hand, this force $\F$ perfectly corresponds to the principal result shown in this article (see Theorem~\ref{the-main}). More precisely, the vector field $\Q\in B(0,\ell)\mapsto 2\De\, (\nabla \u)^T \cdot \Q$ is clearly bounded and we can write the ``explosive'' term~$\E$ as follow:
\begin{equation*}
\E = -\nabla V
\quad \text{with} \quad 
V(\Q) = \frac{\ell^2}{2}\ln \Big( 1-\frac{|\Q|^2}{\ell^2} \Big).
\end{equation*}
To make appear the maxwellian function~$M$ as it is used in this paper, we write
\begin{equation}\label{maxwellian1511}
M(\Q) 
= 
\frac{\dsp e^{V(\Q)}}{\dsp \int_\Omega e^{V({\bf R})}d{\bf R}}
=
\frac{\dsp \Big( 1-\nfrac{|\Q|^2}{\ell^2} \Big)^{\nfrac{\ell^2}{2}}}{\dsp \int_{B(0,\ell)} \Big( 1-\nfrac{|{\bf R}|^2}{\ell^2} \Big)^{\nfrac{\ell^2}{2}}d{\bf R}}.
\end{equation}
From Theorem~\ref{the-main} we deduce that if $\ell>\sqrt{2}$ then for any $\rho\in \R$ there exists a unique (weak) solution of the Fokker-Planck equation~\eqref{FP1436} such that $\int_{B(0,\ell)}\fy = \rho$.
%-------------------------------
According to H.C. Öttinger~\cite{Ottinger}, the number~$\ell$ roughly measures the number of monomer units represented by a bead and it is generally larger than~$10$.
The assumption $\ell>\sqrt{2}$ is not constraining from the physical point of view. In fact, according to H.C. Öttinger~\cite{Ottinger}, the number~$\ell$ roughly measures the number of monomer units represented by a bead and it is generally larger than~$10$.\\
Moreover, impose the quantity $\int_{B(0,\ell)}\fy$ physically corresponds to given the density of the polymer chains. Hence this condition is relevant for the studied problem.
%----------------------------------------
\begin{rema}\label{rem-corot}
If the tensor $(\nabla_\x \u)^T$ is replaced by its anti-symmetric part $\frac{1}{2}( \nabla_\x \u - (\nabla_\x \u)^T)$ in the force term~$\F$ then we get the so-called co-rotational FENE model. This case corresponds to a particular cases presented page~\pageref{some-known}: $\fy=M$ is a trivial solution of equation~\eqref{FP1436} (see~\cite{Chupin, Masmoudi}).
\end{rema}
\subsection{Numerical results}\label{sub3-2}
%%%%%%%%%%%%%%%%%%%%%%%%%%%%%%%%%%%%%%%%%%%%%%%%%%%%%%%%%%%%%%%%%%%%%%%%%%%%%%
%
\noindent In this subsection, we present numerical result for the Fokker-Planck equation~\eqref{FP} for a confinement vector field~$\F$ coupled with the normalization condition $\int_\Omega \fy = \rho$, $\rho \in \R$, and then we apply the algorithm in the framework of fluid mechanics.
The main difficulty to obtain a numerical scheme for the Fokker-Planck equation within the normalized condition is to treat this normalized condition since the equation is not numerically difficult itself. Precisly, this condition is implemented by penalization.
% and we solve the following weak formulation, for~$\eps$ small enough (in practice of order of $10^{-8}$):
%\begin{equation*}
%- \int_\Omega \nabla \fy \cdot \nabla \psi
%+ \int_\Omega \fy ~ \F \cdot \nabla \psi
%= \eps \int_\Omega (\fy-\rho) \psi \quad \text{for all $\psi\in H^1(\Omega)$}.
%\end{equation*}
For simulation, we use the FreeFem++ program\footnote{see http://www.freefem.org/ff++} which is based on weak formulation of the problem and finite elements method.\\[0.3cm]
\noindent In the fluid mechanics context, we want to observe the distribution of the orientation dumbells in a dilute polymer under shear (for instance with a given stationary velocity flow given of the form $\u(x_1,x_2)=(\dot{\gamma}\, x_2,0)$, $\dot{\gamma}\in \R$, in the $2$-dimensional case).
For simplicity, we make the presentation with the $2$-dimensional model. According to the previous subsection, the searched distribution satisfies the following Fokker-Planck equation
\begin{equation}\label{FPnum}
-\Delta \fy + \div(\fy\, \F) = 0 \quad \text{on $B(0,\ell)$},
\end{equation}
where the vector field~$\F$ is given by
\begin{equation*}
\F : 
\begin{pmatrix}
Q_1\\
Q_2\\
\end{pmatrix}
\in B(0,\ell) \longmapsto
2 \mathcal De \, \dot{\gamma}
\begin{pmatrix}
Q_2\\
0
\end{pmatrix}
- \frac{1}{1-\nfrac{|\Q|^²}{\ell^2}}
\begin{pmatrix}
Q_1\\
Q_2\\
\end{pmatrix}.
\end{equation*}
Moreover, the solution must be satisfy the relation $\int_{B(0,\ell)} \fy = \rho$. Notice that if we have a solution such that $\int_{B(0,\ell)} \fy = 1$ then, by linearity, the function $\overline \fy = \rho\fy$ is a solution such that $\int_{B(0,\ell)} \overline \fy = \rho$. In the numerical test, we always take $\rho=1$. The only two parameters which are interest are the product $\mathcal De \, \dot{\gamma}$ and the coefficient~$\ell$ which corresponds to the maximal elongation of the dumbells.\\
Without shear (that is for $\dot{\gamma}=0$), a trivial solution of the Fokker-Planck equation~\eqref{FPnum} exists: it is the maxwellian~$M$ (see its expression~\eqref{maxwellian1511}). For three characteristic maximal lenghts of the dumbells ($\ell=2$, $\ell=5$ and $\ell=10$), we have been represented this maxwellian on the figure~\ref{fig0}.
%
%----------------------------------------
\begin{figure}[htbp]
\begin{center}
\includegraphics[width=2.75cm]{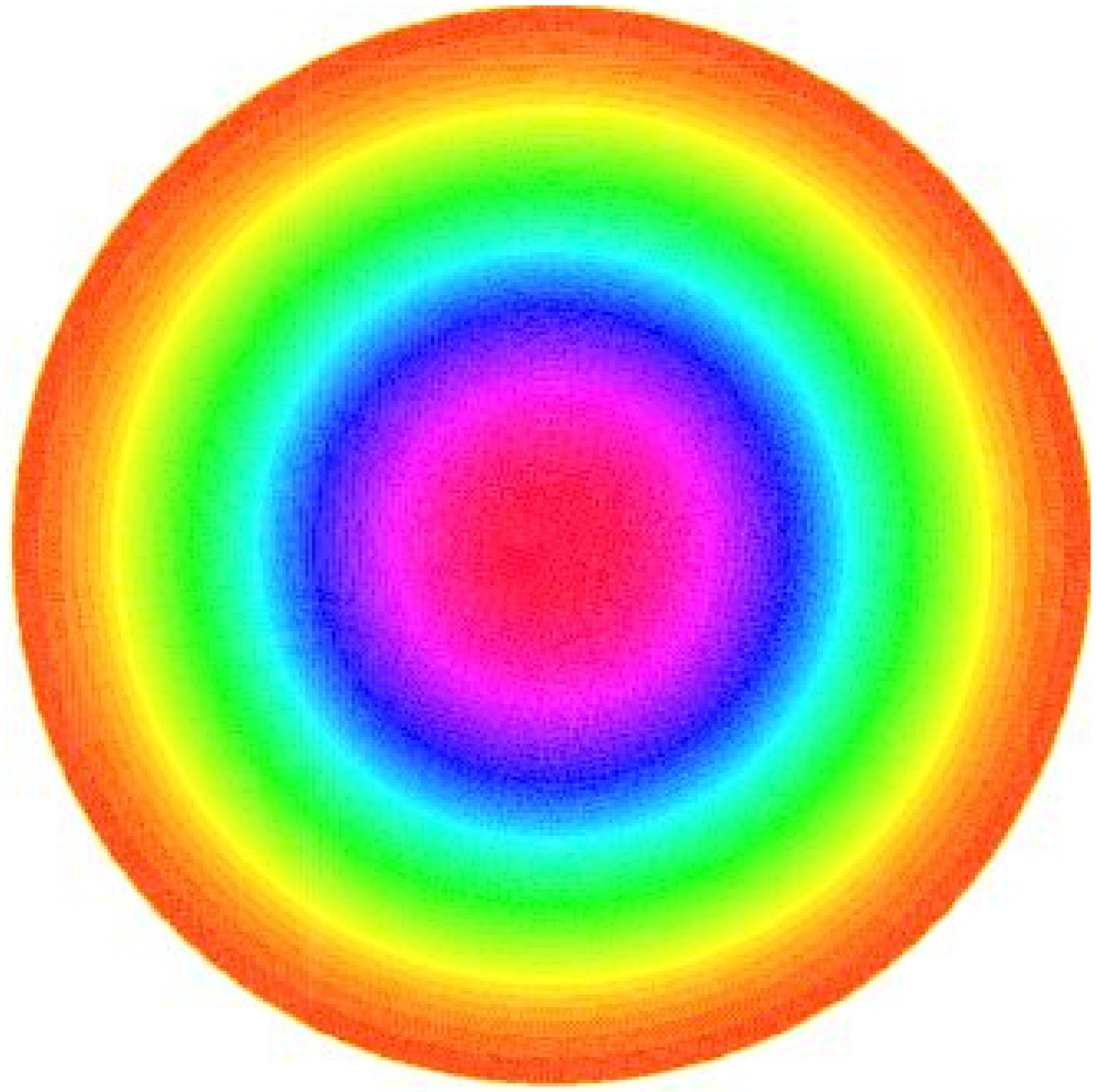}\qquad
\includegraphics[width=2.75cm]{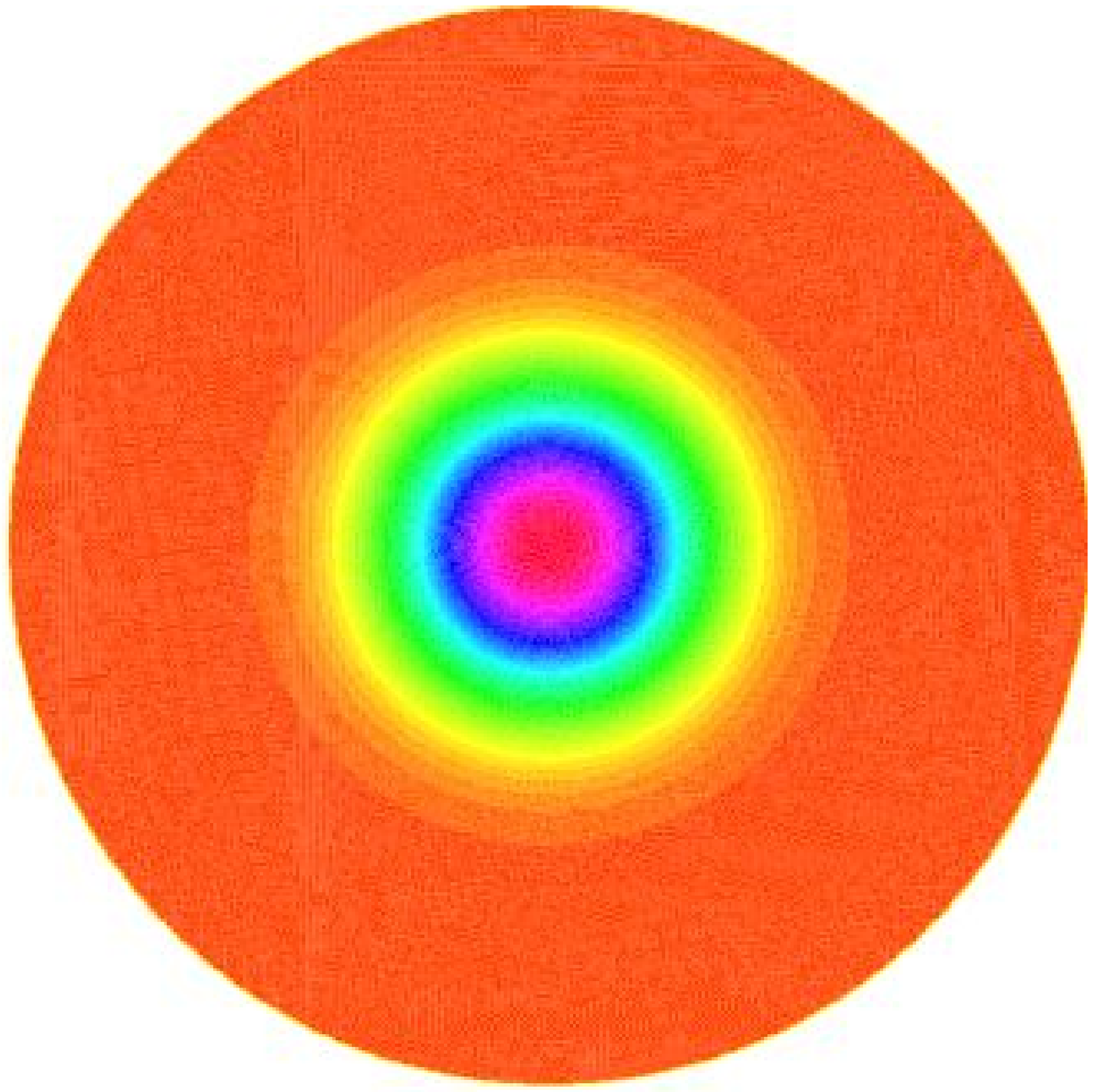}\qquad
\includegraphics[width=2.75cm]{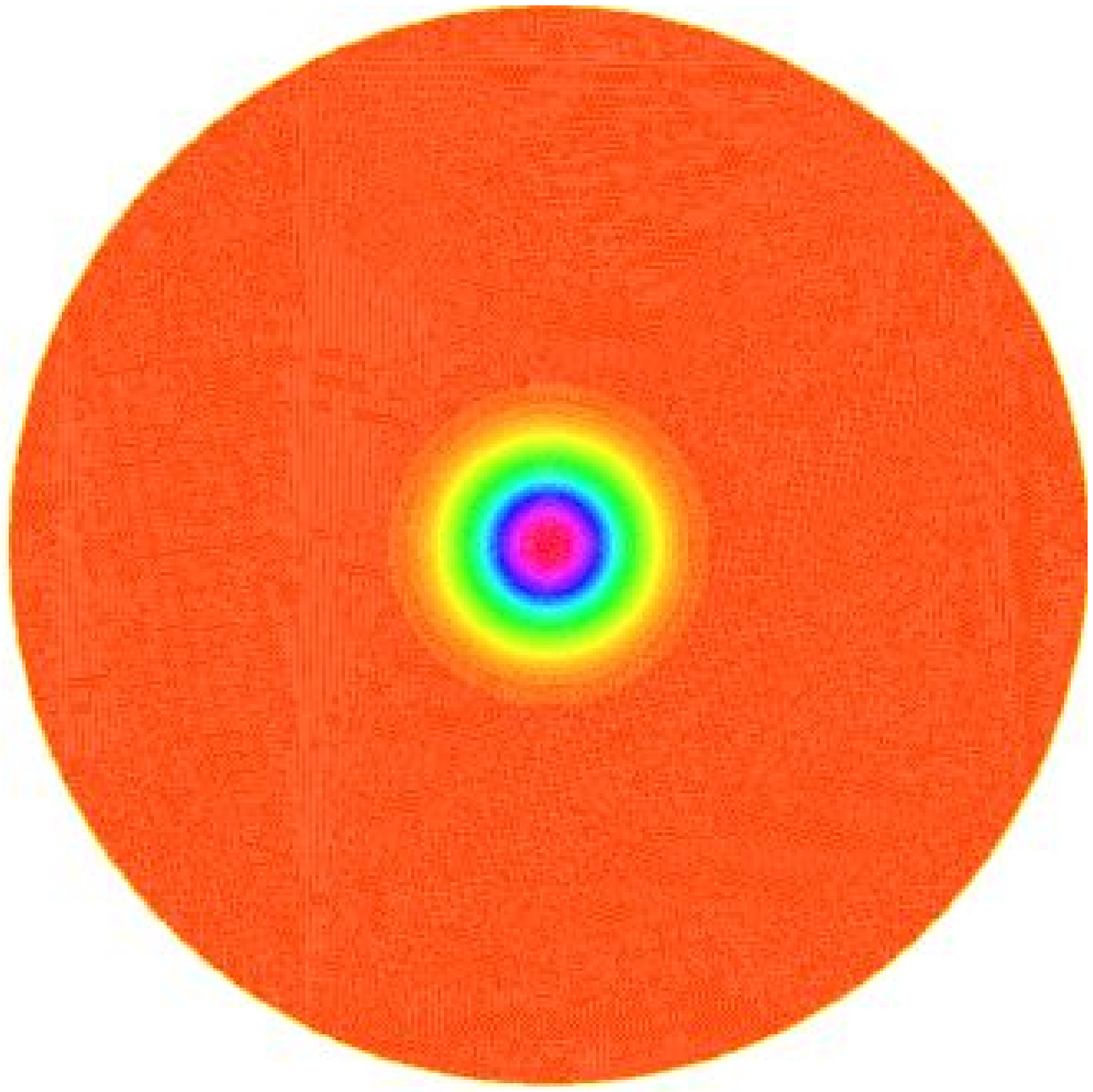}
\caption{Solution without shear for different maximal lengths of dumbells: $\ell=2$, $\ell=5$ and $\ell=10$.}
\label{fig0}
\end{center}
\end{figure}
%----------------------------------------
%

\noindent To observe the influence of the shear on the distribution, taking $\mathcal De=10$, $\ell=5$ and different values of the shear coefficient $\dot{\gamma}\in\{0.1;0.2;0.5;1\}$. The for results are descibe on figure~\ref{fig1}.

%----------------------------------------
\begin{figure}[htbp]
\begin{center}
%\hspace{-1cm}
%\includegraphics[height=3.5cm]{shear01.eps}
\includegraphics[width=2.75cm]{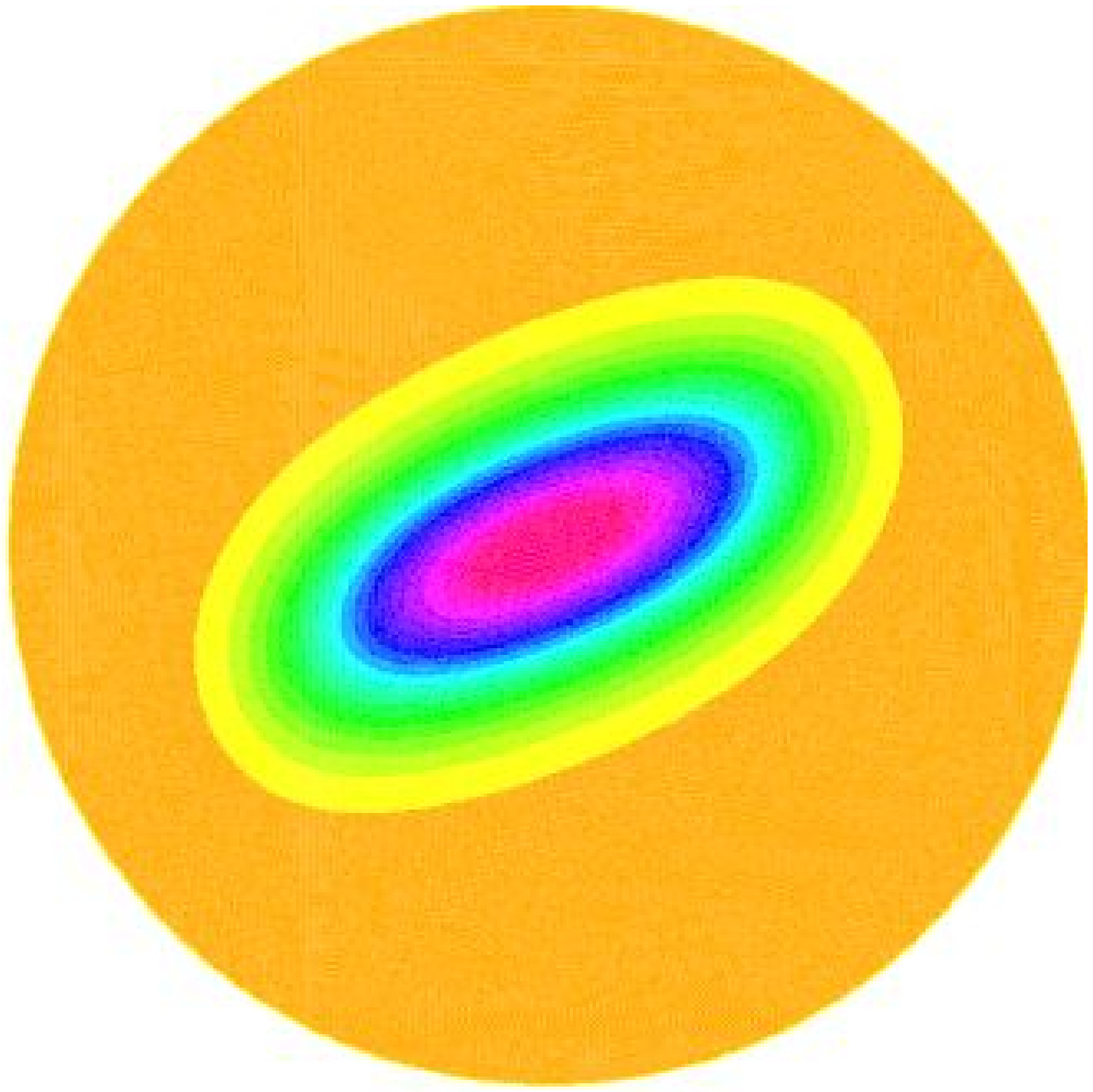}
\includegraphics[width=2.75cm]{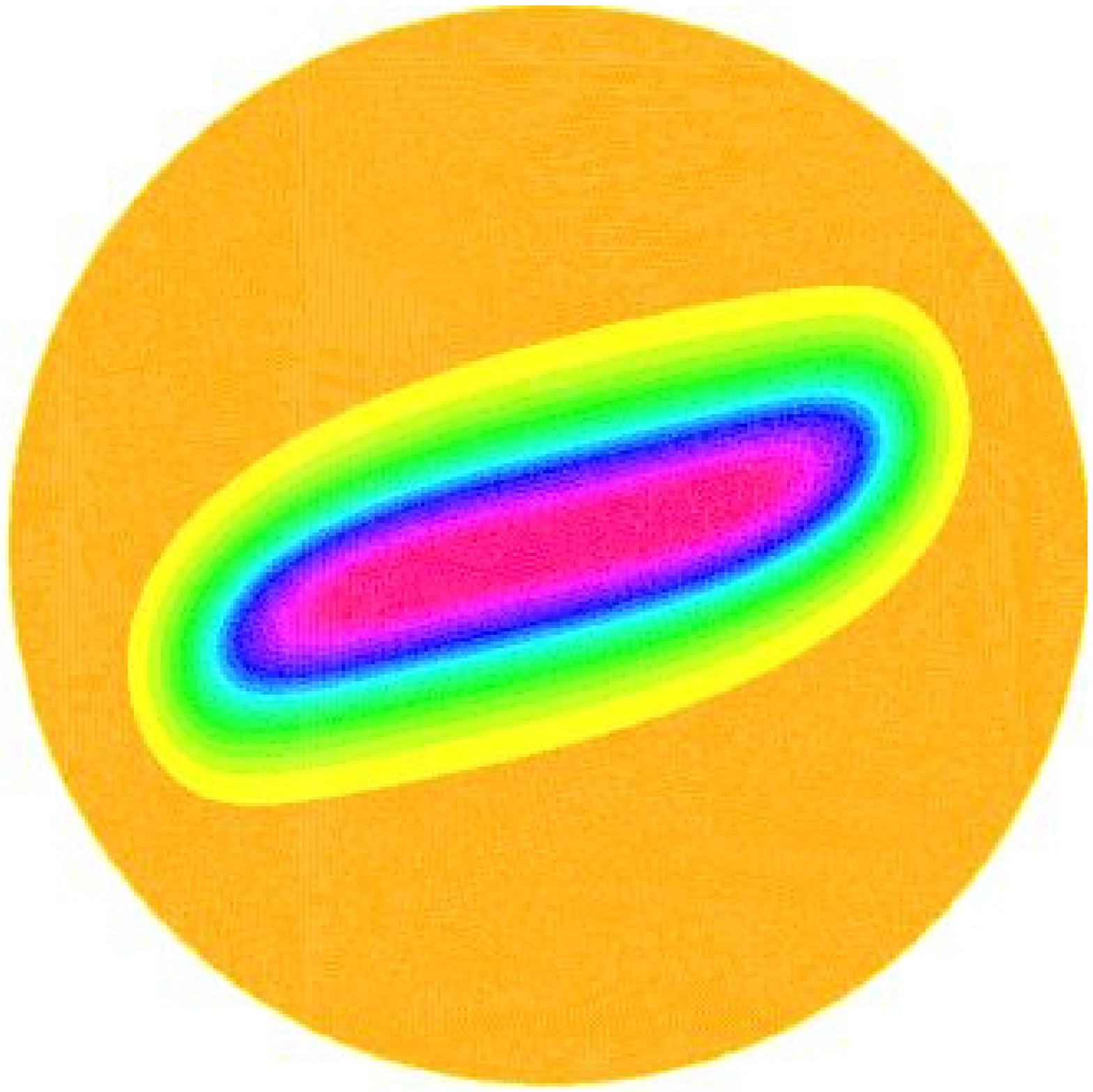}
\includegraphics[width=2.75cm]{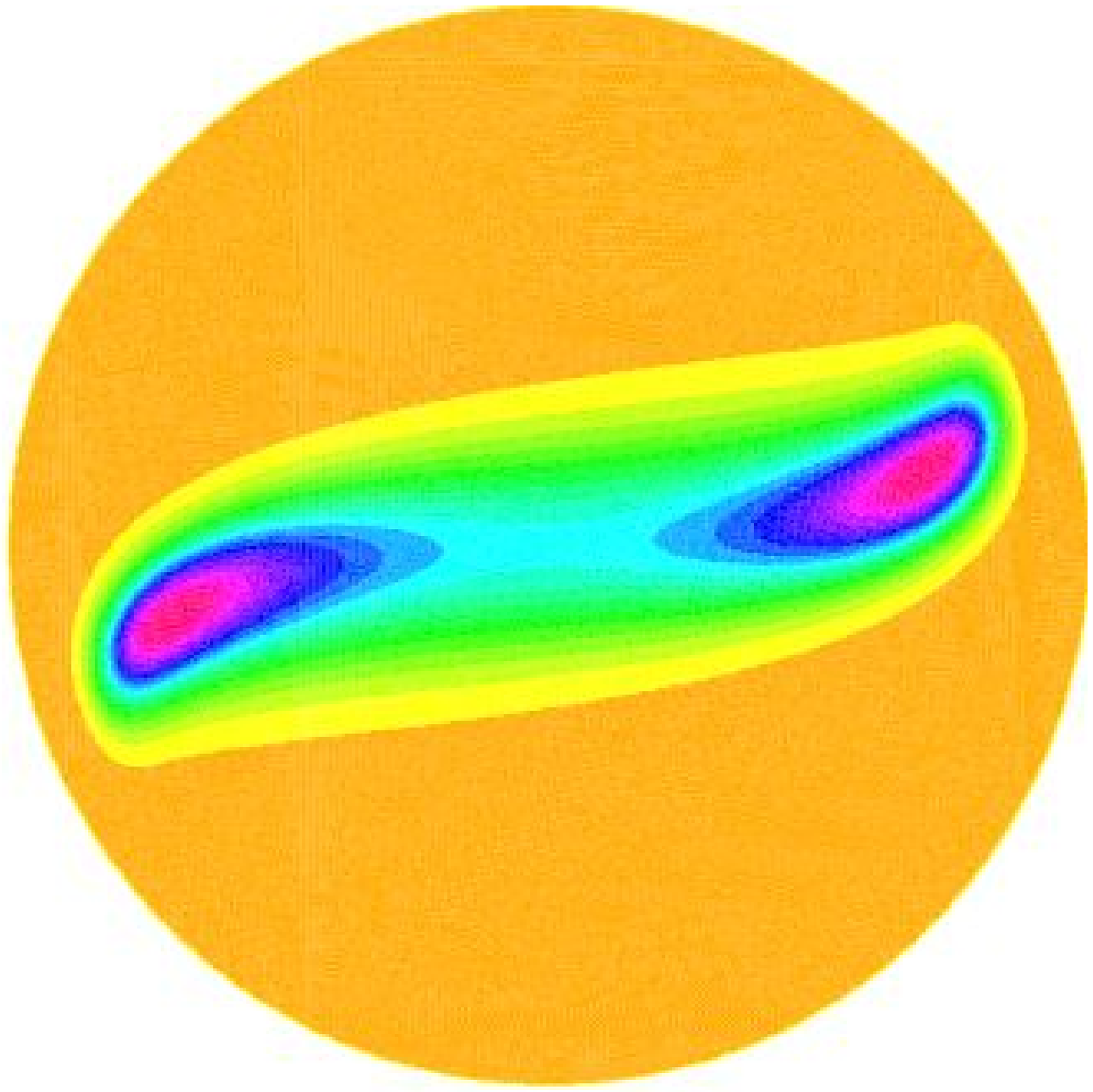}
\includegraphics[width=2.75cm]{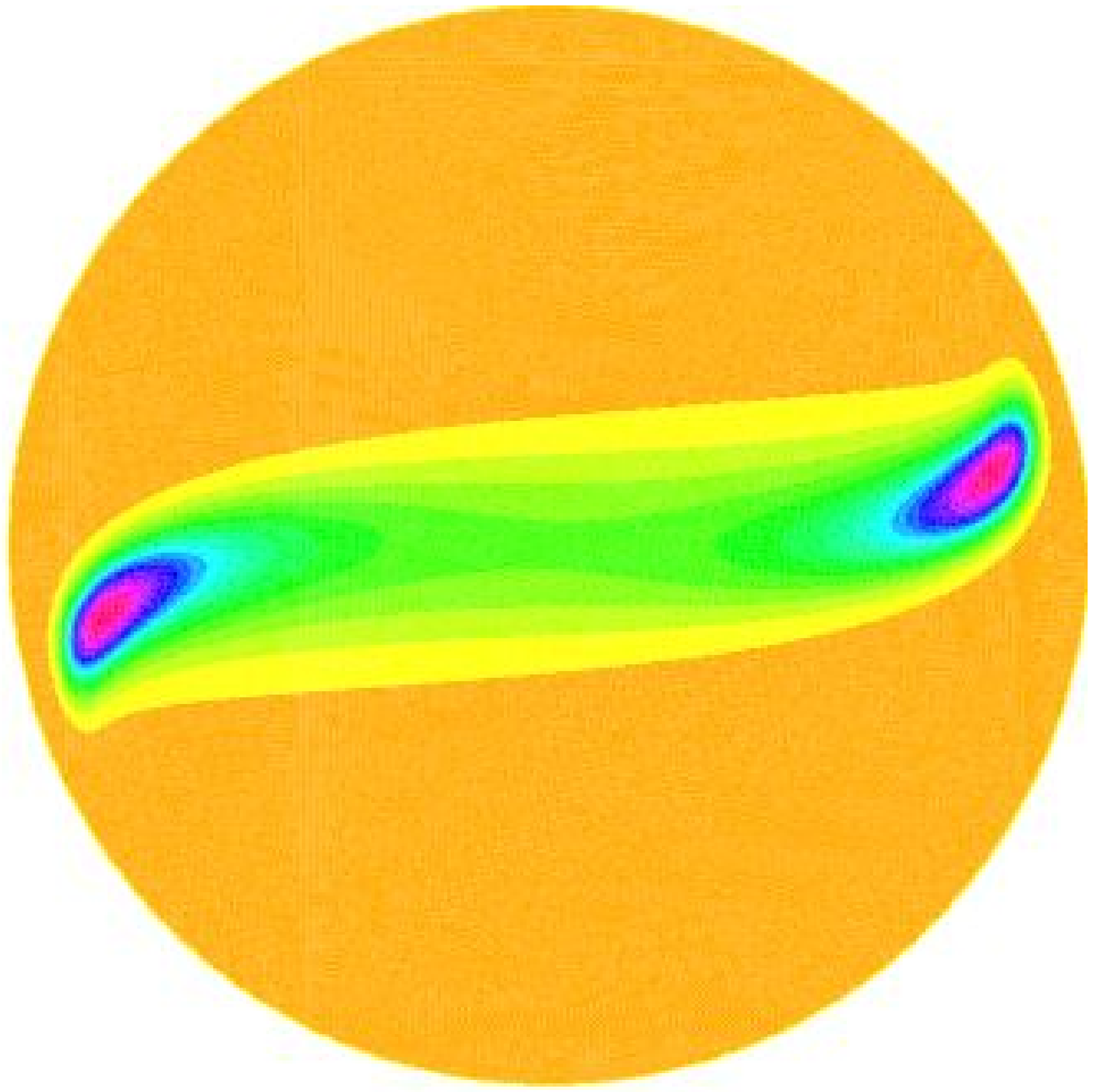}
\caption{
Shear influence on the distribution of the length of dumbells.
The more the shear is raised, the more the dumbbells tend to lengthen in the direction of the flow. The four figures above correspond (from left to right) to the values $\dot{\gamma}=0.1$, $\dot{\gamma}=0.2$, $\dot{\gamma}=0.5$ and $\dot{\gamma}=1$.
}
\label{fig1}
\end{center}
\end{figure}
%----------------------------------------

%Parler des modèle multi-ressorts ?\cite{Ghosh}\\
%----------------------------------------

\nocite{*}
\bibliographystyle{cdraifplain}
\bibliography{biblio}
\end{document}